\documentclass[11pt,reqno]{amsart}
\usepackage{amssymb,mathrsfs,graphicx,subfigure, enumerate}
\usepackage{amsmath,amsfonts,amssymb,amscd,amsthm,bbm}
\usepackage{extpfeil}
\usepackage{graphicx,colortbl}
\usepackage{float}
\usepackage{epsfig}
\usepackage{caption}
\usepackage{bm}
\graphicspath{{Figure/}}

\usepackage[margin=1in]{geometry}

\makeatletter
\@namedef{subjclassname@2020}{\textup{2020} Mathematics Subject Classification}
\makeatother

\title[Hydrodynamic limit of the Schr\"odinger-Chern-Simons system]{Hydrodynamic limits of the nonlinear Schr\"odinger equation
with the Chern-Simons gauge fields}

\author[Kim]{Jeongho Kim}
\address[Jeongho Kim]{\newline Department of Mathematics, Research Institute for Natural Sciences, \newline Hanyang University, Seoul 04763, Republic of Korea}
\email{jeonghokim206@gmail.com}

\author[Moon]{Bora Moon}
\address[Bora Moon]{\newline Department of Mathematics, Research Institute for Natural Sciences, \newline Hanyang University, Seoul 04763, Republic of Korea}
\email{boramoon@hanyang.ac.kr}
\email{boramoon@outlook.kr}

\begin{document}
	\newtheorem{theorem}{Theorem}[section]
	\newtheorem{lemma}{Lemma}[section]
	\newtheorem{corollary}{Corollary}[section]
	\newtheorem{proposition}{Proposition}[section]
	\newtheorem{remark}{Remark}[section]
	\newtheorem{definition}{Definition}[section]
	
	\renewcommand{\theequation}{\thesection.\arabic{equation}}
	\renewcommand{\thetheorem}{\thesection.\arabic{theorem}}
	\renewcommand{\thelemma}{\thesection.\arabic{lemma}}
	\newcommand{\bbr}{\mathbb R}
	\newcommand{\bbz}{\mathbb Z}
	\newcommand{\bbn}{\mathbb N}
	\newcommand{\bbs}{\mathbb S}
	\newcommand{\bbp}{\mathbb P}
	\newcommand{\bbt}{\mathbb T}
	\newcommand{\ddiv}{\textrm{div}}
	\newcommand{\bn}{\bf n}
	\newcommand{\rr}[1]{\rho_{{#1}}}
	\newcommand{\thh}{\theta}
	\def\charf {\mbox{{\text 1}\kern-.24em {\text l}}}
	\renewcommand{\arraystretch}{1.5}
	
	%%%%%%%%%%%%%%%%%%%%%%%%%%%
	\newcommand{\T}{\mathbb{T}}
	\newcommand{\N}{\mathbb{N}}
	\newcommand{\R}{\mathbb{R}}
	\newcommand{\lt}{\left}
	\newcommand{\rt}{\right}
	\newcommand{\bq}{\begin{equation}}
	\newcommand{\eq}{\end{equation}}
	\newcommand{\e}{\varepsilon}
	\newcommand{\mc}{\mathcal{C}}
	\newcommand{\pa}{\partial}
	\newcommand{\ph}{\hat{p}}
	\renewcommand{\d}{\textup{d}}
	\renewcommand{\i}{\textup{i}}
	\renewcommand{\Re}{\textup{Re}}
	\renewcommand{\Im}{\textup{Im}}
	%%%%%%%%%%%%%%%%%%%%%%%%%%%%%%
	
	
	
	\thanks{The work of J. Kim and B. Moon was supported by the Basic Science Research Program through the National Research Foundation of Korea(NRF) funded by the Ministry of Science and ICT (NRF-2020R1A4A3079066) and the work of B. Moon was supported by Basic Science Research Program through the National Research Foundation of Korea(NRF) funded by the Ministry of Education(2019R1I1A1A01059585)}

	\begin{abstract} 
		We present two types of the hydrodynamic limit of the nonlinear Schr\"odinger-Chern-Simons (SCS) system. We consider two different scalings of the SCS system and show that each SCS system asymptotically converges towards the compressible and incompressible Euler system, coupled with the Chern-Simons equations and Poisson equation respectively, as the scaled Planck constant converges to 0. Our method is based on the modulated energy estimate. In the case of compressible limit, we observe that the classical theory of relative entropy method can be applied to show the hydrodynamic limit, with the additional quantum correction term. On the other hand, for the incompressible limit, we directly estimate the modulated energy to derive the desired asymptotic convergence.
	\end{abstract}
	
	\maketitle
	
	%\tableofcontents
	
	\section{Introduction}\label{sec:1}
	\setcounter{equation}{0}
	
	In this paper, we consider the hydrodynamic limit of the Schr\"odinger-Chern-Simons system with the general potential:
	\begin{align}
	\begin{aligned}\label{Chern-Simons-Schrodinger}
	&\i \hbar D_0\psi +\frac{\hbar^2}{2m}(D_1D_1\psi+D_2D_2\psi) -V'(|\psi|^2)\psi=0,\quad (t,x)\in \bbr_+\times \Omega,\\
	&\partial_0 A_1 -\partial_1A_0 = -\hbar\Im(\overline{\psi}D_2\psi),\quad \partial_0 A_2-\partial_2A_0 = \hbar\Im(\overline{\psi}D_1\psi),\\
		&\partial_1 A_2-\partial_2A_1 = -m|\psi|^2,
	\end{aligned}	
	\end{align}
	subject to initial data 
	\[(\psi,A_\mu)|_{t=0} = (\psi_{\textup{in}},A_{\mu,in}),\]
	where $\i := \sqrt{-1}$, $\pa_0 := \frac{\pa}{\pa t}$, $\pa_i := \frac{\pa}{\pa x_i}$, $\psi:\bbr_{\geq 0}\times \Omega\to \mathbb{C}$ is a complex scalar field, $m$ is the mass of the particle, $A_\mu:\bbr_{\geq 0}\times\Omega\to\bbr$ is the gauge field and $V$ is the potential energy density of the fields. Here, the spatial domain is either the whole domain $\Omega=\bbr^2$ or the periodic domain $\Omega=\bbt^2$. Also, $D_\mu = \partial_\mu +\frac{\i}{\hbar}A_\mu$ is the covariant derivative for $\mu = 0,1,2$. The SCS system \eqref{Chern-Simons-Schrodinger} proposed in \cite{EHI91,JP90-1,JP90} consists of the Schrödinger equation augmented by the gauge field $A_{\mu}$ to study vortex solitons in a nonrelativistic (2+1)-dimensional Chern-Simons gauge theory. It is well-known that the SCS system \eqref{Chern-Simons-Schrodinger} is invariant under the following gauge transformation:
	\[\psi\to \psi e^{\i\chi},\quad A_\mu \to A_\mu -\hbar\pa_\mu \chi.\]
	Therefore, to determine the gauge field, one need to give an extra condition on the gauge field. In many cases, the following Coulomb gauge condition is assumed:
	\[\nabla\cdot A = \pa_1 A_1+\pa_2A_2 = 0,\quad A:=(A_1,A_2).\]
	
	Under the Coulomb gauge condition, the Cauchy problem for the SCS system \eqref{Chern-Simons-Schrodinger} is equivalent to the following system: 
	\begin{align}
		\begin{aligned}\label{CSS-coulomb-2}
			&\i\hbar\partial_t \psi -A_0\psi +\frac{\hbar^2}{2m}\left(\Delta \psi +\frac{2\i}{\hbar}A\cdot\nabla \psi - \frac{1}{\hbar^2}|A|^2\psi\right)-V'(|\psi|^2)\psi=0,\\
			%&\pa_t A_1 -\pa_1 A_0 = -\hbar\Im(\overline{\psi}D_2\psi),\quad \pa_t A_2-\pa_2A_0 = \hbar\Im(\overline{\psi}D_1\psi),\\
			&\Delta A_0 = \hbar\Im(Q_{12}(\overline{\psi},\psi))+\partial_1(A_2|\psi|^2)-\partial_2(A_1|\psi|^2),\\
			&\Delta A_1= m\partial_2 |\psi|^2,\quad \Delta A_2=-m\partial_1|\psi|^2,
		\end{aligned}
	\end{align}
	where $Q_{12}(\overline{\psi},\psi):=\pa_1\overline{\psi}\pa_2\psi -\pa_2\overline{\psi}\pa_1\psi$, when the initial data $(\psi_{\textup{in}}, A_{\textup{in}})$ satisfy the following compatibility conditions
	\begin{equation}\label{compatibility}
		\pa_1 A_{2,\textup{in}}-\pa_2 A_{1,\textup{in}} = -m|\psi_{\textup{in}}|^2,\quad \nabla\cdot A_{\textup{in}}=0.
	\end{equation}

	\vspace{0.3cm}
	\noindent We choose $m=1$ and $\hbar = \e$ in \eqref{CSS-coulomb-2} to obtain the following scaled SCS equation:
	\begin{align}
	\begin{aligned}\label{CSS-e}
	&\i\e\partial_t \psi^\e -A^\e_0\psi^\e +\frac{\e^2}{2}\Delta \psi^\e +\i\e A^\e\cdot\nabla \psi^\e - \frac{1}{2}|A^\e|^2\psi^\e-V'(|\psi^\e|^2)\psi^\e=0,\\
	&\Delta A^\e_0 = \e\Im(Q_{12}(\overline{\psi^\e},\psi^\e))+\partial_1(A^\e_2|\psi^\e|^2)-\partial_2(A^\e_1|\psi^\e|^2),\\
	&\Delta A^{\e}_1= \partial_2 |\psi^{\e}|^2,\quad \Delta A^\e_2=-\partial_1|\psi^{\e}|^2,
	\end{aligned}
	\end{align}
	subject to the initial data $(\psi^\e_{\textup{in}},A^\e_{\textup{in}})$ satisfying \eqref{compatibility}.  We remind that, the compatibility condition \eqref{compatibility} propagates along the system $\eqref{Chern-Simons-Schrodinger}_{1,2}$:	  
	\begin{equation}\label{additional}
		\nabla\times A^\e = \pa_1 A^\e_2- \pa_2 A^\e_1 = -|\psi^\e|^2,\quad \nabla\cdot A^\e=0,\quad t\ge0,
	\end{equation}
and the additional relations correspond to $\eqref{Chern-Simons-Schrodinger}_{2}$ also hold:
\begin{align}\label{additional_2}
	\begin{aligned}	
		\pa_t A^\e_1 -\pa_1 A^\e_0 = -\e\Im(\overline{\psi^\e}D^\e_2\psi^\e),\quad \pa_t A^\e_2-\pa_2A^\e_0 = \e\Im(\overline{\psi^\e}D^\e_1\psi^\e).
	\end{aligned}
\end{align}

\vspace{0.3cm}

	\noindent The main concern of the paper is the hydrodynamic limit of \eqref{CSS-e} in the semiclassical regime, i.e., when $\hbar=\e\to0$. We consider two different scaling regimes, each of which leads to the compressible and incompressible Euler-type equations as an asymptotic system. We first focus on the hydrodynamic limit towards the compressible Euler equations. In this case, we consider the special family of the potential $V(\rho):=\frac{1}{\gamma}\rho^\gamma$ with $\gamma\ge1$. This choice of the potential covers the standard quadratic potential $V(\rho)=\frac{1}{2}\rho^2$. To identify the asymptotic hydrodynamic system, we consider the following Madelung transformation \cite{M27}:
	\[\psi^\e(t,x)=\sqrt{\rho^\e(t,x)}\exp\left(\frac{\i}{\e} S^\e(t,x)\right).\]
	Then, the hydrodynamic quantities $\rho^\e = |\psi^\e|^2$ and $u^\e:=\nabla S^\e +A^\e = \frac{\i \e}{2|\psi^\e|^2}(\psi^\e\nabla\overline{\psi^\e}-\overline{\psi^\e}\nabla\psi^\e)+A^\e$ satisfy the following compressible Euler-type equations coupled with the Chern-Simons equations:
	\begin{align}
	\begin{aligned}\label{CSS-hydro-e}
	&\partial_t \rho^\e +\nabla\cdot(\rho^\e u^\e) = 0,\\
	&\partial_t (\rho^\e u^\e) +\nabla\cdot(\rho^\e u^\e\otimes u^\e) +\nabla p(\rho^\e) =\rho^\e\frac{\e^2}{2}\nabla\left(\frac{\Delta \sqrt{\rho^\e}}{\sqrt{\rho^\e}}\right),\\
	&\Delta A^\e_0 = \nabla\times(\rho^\e u^\e),\quad \Delta A^\e = -(\nabla \rho^\e)^\perp,
	\end{aligned}
	\end{align}
 subject to initial data $(\rho^{\e}_{\textup{in}}, A^{\e}_{\textup{in}})$ satisfying the compatibility condition
 \begin{align*}
 	\nabla\times A^{\e}_{\textup{in}}=-\rho^\e_{\textup{in}},\quad \nabla\cdot A^{\e}_{\textup{in}}=0.
 \end{align*}
	Here, $p(\rho):=\rho V'(\rho)-V(\rho)=\frac{\gamma-1}{\gamma}\rho^\gamma$ is a pressure and $u^\perp:=(-u_2,u_1)$. The density-dependent source term in the momentum equation \eqref{CSS-hydro-e}$_2$ is so-called the Bohm potential, which is exactly the same as $\frac{\e^2}{4}\nabla\cdot(\rho^\e\nabla^2\log \rho^\e)$
	and this term is also understood as a quantum pressure tensor. We refer Section \ref{sec:2.2.1} for the detailed procedure to derive \eqref{CSS-hydro-e}. Heuristically, the hydrodynamic equations \eqref{CSS-hydro-e} converges to the following compressible Euler-Chern-Simons (ECS) system as $\e\to0$:
	\begin{align}
	\begin{aligned}\label{CSE}
	&\partial_t \rho +\nabla\cdot(\rho u) = 0,\\
	&\partial_t (\rho u) +\nabla\cdot(\rho u\otimes u) +\nabla p(\rho) =0,\\
	&\Delta A_0 = \nabla\times( \rho u),\quad \Delta A=-(\nabla \rho)^{\perp},
	\end{aligned}
	\end{align}
	subject to initial data $(\rho_{\textup{in}},A_{\textup{in}})$ satisfying the compatibility condition
	\begin{equation}\label{compatibility-CE}
	\nabla\times A_{\textup{in}} = -\rho_{\textup{in}},\quad \nabla\cdot A_{\textup{in}} = 0.
	\end{equation}
	The first goal of the paper is to obtain the rigorous convergence from \eqref{CSS-hydro-e} to \eqref{CSE}.\\
	
	On the other hand, for the case of incompressible limit, we further scale the equation \eqref{CSS-e} by taking $t\to \e^{\alpha}t$, $A_0^\e\to \e^\alpha A_0^\e$ for $0<\alpha<1$ and $A^\e\to\e^{-\alpha}A^\e$ and consider another special family of potential $V(\rho) = \frac{1}{\gamma}(\rho^{\frac{\gamma}{2}}-1)^2$ with $\gamma\ge2$. Then, the rescaled SCS system becomes:
	\begin{align}
		\begin{aligned}\label{CSS-e-incomp}
			&\i\e^{\alpha+1} \pa_t\psi^\e - \e^\alpha A_0^\e\psi^\e+\frac{\e^2}{2}\Delta\psi^\e+\i\e^{1-\alpha} A^\e\cdot \nabla \psi^\e -\frac{\e^{-2\alpha}}{2}|A^\e|^2\psi^\e-(|\psi^\e|^\gamma-1)|\psi^\e|^{\gamma-2}\psi^\e=0,\\
			&\Delta A_0^\e = \e^{1-\alpha}\Im(Q_{12}(\overline{\psi^\e},\psi^\e))+\e^{-2\alpha}\pa_1(A_2^\e|\psi^\e|^2)-\e^{-2\alpha}\pa_2(A_1^\e|\psi^\e|^2),\\
			&\Delta A^{\e}_1=\e^{\alpha}\pa_2 |\psi^\e|^2,\quad \Delta A^{\e}_2=-\e^{\alpha}\pa_1 |\psi^\e|^2,
		\end{aligned}
	\end{align}
	 Again, considering the same Madelung transformation and introduce the hydrodynamic quantities $\rho^\e:=|\psi^\e|^2$ and $u^\e := \e^{-\alpha}\nabla S^\e +\e^{-2\alpha} A^\e=\frac{\i \e^{1-\alpha}}{2|\psi^\e|^2}(\psi^\e\nabla\overline{\psi^\e}-\overline{\psi^\e}\nabla\psi^\e)+\e^{-2\alpha}A^\e$, one can derive the similar Euler-type equations coupled with Chern-Simons equations:
	\begin{align}
		\begin{aligned}\label{CSS-hydro-e-2}
		&\partial_t \rho^\e +\nabla\cdot(\rho^\e u^\e) = 0,\\
		&\partial_t (\rho^\e u^\e) +\nabla\cdot(\rho^\e u^\e\otimes u^\e) +\e^{-2\alpha}\nabla p(\rho^\e) =\rho^\e\frac{\e^{2-2\alpha}}{2}\nabla\left(\frac{\Delta \sqrt{\rho^\e}}{\sqrt{\rho^\e}}\right),\\
		&\Delta A^\e_0 = \nabla\times(\rho^\e u^\e),\quad \Delta A^\e=-\e^\alpha(\nabla \rho^\e)^{\perp},
		\end{aligned}
	\end{align}
	where the pressure is now $p(\rho) = \frac{\gamma-1}{\gamma}\rho^\gamma-\frac{\gamma-2}{\gamma}\rho^{\frac{\gamma}{2}}$. Heuristically, it follows from \eqref{CSS-hydro-e-2}$_2$ that $\nabla p(\rho^\e)= \nabla\left(\frac{\gamma-1}{\gamma}(\rho^\e)^\gamma-\frac{\gamma-2}{\gamma}(\rho^\e)^{\frac{\gamma}{2}}\right)$ is dominating term, and therefore, as $\e \to 0$, we may guess $\rho^\e$ converges to constant. Consequently, an appropriate limit system for \eqref{CSS-hydro-e-2} would be the following incompressible Euler equations coupled with the Poisson equation:
	\begin{align}
		\begin{aligned}\label{CSE-incomp}
			&\partial_t u +(u\cdot \nabla)u +\nabla \pi =0,\quad \nabla\cdot u =0,\\
			&\Delta A_0 = \nabla\times u.
		\end{aligned}
	\end{align}
	The second goal of the paper is the rigorous derivation of \eqref{CSE-incomp} from \eqref{CSS-hydro-e-2}. \\

	The rest of the paper is organized as follows. In Section \ref{sec:2}, we present the preliminaries for the Schr\"odinger-Chern-Simons equations and its hydrodynamic formulation. We also state the main theorem of the paper in this section. Section \ref{sec:3} provides the relation between two concepts, the relative entropy used in the kinetic theory and the modulated energy used in the quantum theory. Indeed, we show that they are almost equivalent concepts, except for the quantum correction term. Then, we provide the complete proofs of Theorem \ref{thm:main-1} in Section 4, using the classical relative entropy method and the relation between modulated energy and the relative entropy. In Section 5, we provide the proof of Theorem \ref{thm:main-2} by directly estimating modulated energy. Finally, Section \ref{sec:6} is devoted to a summary and future perspectives. 

	\section{Preliminaries}\label{sec:2}
	\setcounter{equation}{0}
	In this section, we present the basic conservation laws for SCS equations \eqref{CSS-e} and \eqref{CSS-e-incomp}, and identify the asymptotic hydrodynamic equations. We also provide the main theorem of the paper at the end of the section.
	
	\subsection{Conservation laws}
	One of the most important properties of the SCS system \eqref{Chern-Simons-Schrodinger} is the conservation of charge and energy. 
	One can obtain the conservation laws for  SCS equations \eqref{CSS-e} and \eqref{CSS-e-incomp}, respectively, using the following lemma.
	
	\begin{lemma} \label{L2.1}
		Let $(\psi,A_0, A)$ be a smooth solution to the SCS system \eqref{Chern-Simons-Schrodinger} subject to the initial data $(\psi_{\textup{in}}, A_{\textup{in}})$. Then, the total charge and the total energy are conserved:
		\begin{align*}
			\begin{aligned}
				&\frac{\d}{\d t}\int_{\Omega} |\psi(t,x)|^2\,\d x=0,\quad
				 \frac{\d}{\d t}\int_{\Omega} \frac{\hbar^2
				 }{2}\sum_{j=1}^2 |D_j\psi(t,x)|^2 +V(|\psi(t,x)|^2)\,\d x=0.
			\end{aligned}
		\end{align*}	
	\end{lemma}
	
	\begin{proof}
	We first note that the covariant derivative satisfies the following calculation rules:
	\begin{align}\label{rule}
		\pa_{\mu}(\bar{\phi}\psi)=\bar{\phi}D_{\mu}\psi+\overline{D_{\mu}\phi}\psi,\quad
		D_{\mu}D_{\nu}\psi= D_{\nu}D_{\mu}\psi+\frac{i}{\hbar}(\pa_{\mu}A_{\nu}-\pa_{\nu}A_{\mu}) \psi.
	\end{align}	
	
	\noindent $\bullet$ (Estimate for charge): We multiply \eqref{Chern-Simons-Schrodinger} by $\bar{\psi}$ and take its imaginary part to derive

\begin{align}\label{P-1}
	\hbar \pa_t|\psi|^2 + \hbar^2\sum_{j=1} \pa_j \mbox{Im}(\bar{\psi}D_j\psi)=0.
\end{align}
By integrating \eqref{P-1} over $\Omega$, we obtain the desired result.
	
	\vspace{0.3cm}
	\noindent $\bullet$ (Estimate for energy):  Similarly, we multiply \eqref{Chern-Simons-Schrodinger} by $\overline{D_0\psi}$ and consider  its real part to observe
	\begin{align}\label{P-2}
	\hbar^2\sum_{j=1}^2\pa_j\mbox{Re}(\overline{D_0\psi}D_j\psi)-\hbar^2\sum_{j=1}^2 \mbox{Re}(\overline{D_jD_0\psi} D_j\psi)-\pa_t\big(V(|\psi|^2)\big)=0.
	\end{align}	
We use \eqref{rule} and $\eqref{Chern-Simons-Schrodinger}_{2}$ to estimate the second term as
\begin{align*}
	\begin{aligned}
		\sum_{j=1}^2 \mbox{Re}(\overline{D_jD_0\psi}D_j\psi)=\sum_{j=1}^2\Big( \mbox{Re}(\overline{D_0D_j\psi}D_j\psi)
		+\frac{1}{\hbar}(\pa_0A_j-\pa_jA_0)\mbox{Im}(\bar{\psi}D_j\psi)\Big)=\pa_t\sum_{j=1}^2\frac{1}{2} |D_j\psi|^2.
	\end{aligned}
\end{align*}
Inserting the above identity to \eqref{P-2} and integrating it over $\Omega$, we conclude the desired estimate. 
\end{proof}
We define natural energy functions for the scaled SCS system \eqref{CSS-e}	and \eqref{CSS-e-incomp}, respectively, as 
\begin{align*}
	\mathcal{E}^\e(t,x)&:=\frac{\e^2}{2}\sum_{j=1}^2|D_j^\e\psi^\e(t,x)|^2+V(|\psi^{\e}(t,x)|^2) := \frac{\e^2}{2}| D^\e\psi^\e(t,x)|^2+\frac{1}{\gamma}|\psi^{\e}|^{2\gamma},\\
	\widetilde{\mathcal{E}}^\e(t,x)&:=\frac{\e^{2-2\alpha}}{2}\sum_{j=1}^2|\widetilde{D}^{\e}_j\psi^\e(t,x)|^2+\e^{-2\alpha}\widetilde{V}(|\psi^\e|^2):=\frac{\e^{2-2\alpha}}{2}|\widetilde{D}^\e\psi^\e(t,x)|^2+\frac{\e^{-2\alpha}}{\gamma}(|\psi^\e|^\gamma-1)^2,
\end{align*}
where  $D^\e_j:=\pa_j+\frac{\i}{\e}A^{\e}_j$ and $\widetilde{D}_j^\e :=\pa_j+\frac{\textup{i}}{\e^{1+\alpha}}A_j^{\e}$. By directly applying Lemma \ref{L2.1} to each scaled SCS system \eqref{CSS-e} and \eqref{CSS-e-incomp} respectively, one can obtain the following conservation laws, written in terms of the corresponding total charge and energy. We omit the detailed proof.

	\begin{proposition}\label{P2.1}
		Let $(\psi^\e,A^\e_0, A^\e)$ be a global solution to the SCS equations \eqref{CSS-e} subject to the initial data $(\psi^\e_{\textup{in}}, A^\e_{\textup{in}})$. Then, the total charge and the total energy are conserved:
		\[\int_{\Omega}|\psi^\e(t,x)|^2\,\d x=\int_{\Omega}|\psi^\e_{\textup{in}}(x)|^2\,\d x,\quad \int_{\Omega}\mathcal{E}^\e(t,x)\,\d x = \int_{\Omega}\mathcal{E}^\e(0,x)\,\d x,\quad t\ge0.\]
	\end{proposition}

	\begin{proposition}\label{P2.2}
		Let $(\psi^\e,A^\e_0,A^\e)$ be a global solution to the SCS equations \eqref{CSS-e-incomp} subject to the initial data $(\psi^\e_{\textup{in}}, A^\e_{\textup{in}})$. Then, the total charge and the total energy are conserved:
		\[\int_{\Omega}|\psi^\e(t,x)|^2\,\d x=\int_{\Omega}|\psi^\e_{\textup{in}}(x)|^2\,\d x,\quad \int_{\Omega}\widetilde{\mathcal{E}}^\e(t,x)\,\d x = \int_{\Omega}\widetilde{\mathcal{E}}^\e(0,x)\,\d x,\quad t\ge0.\]
	\end{proposition}

	\subsection{Identifying the hydrodynamic equations}
	We now present the detailed derivation of the hydrodynamic system for the SCS equations. 
	\subsubsection{Derivation of \eqref{CSS-hydro-e}}\label{sec:2.2.1}
	As we discussed in the introduction, we consider the well-known Madelung transformation for $\psi^\e$:
	\begin{equation}\label{madelung}
	\psi^\e(t,x) = \sqrt{\rho^\e(t,x)} \exp\left(\frac{\i }{\e}S^\e(t,x)\right)
	\end{equation}
	and then substitute the representation \eqref{madelung} of $\psi^\e$ into \eqref{CSS-e}. Then, the imaginary part of the Schr\"odinger equation \eqref{CSS-e}$_1$ becomes
	\begin{equation*}\label{eq-rho}
	\partial_t\sqrt{\rho^\e} +\nabla\sqrt{\rho^\e}\cdot\nabla S^\e +\frac{\sqrt{\rho^\e}}{2}\Delta S^\e +A^\e\cdot \nabla \sqrt{\rho^\e}=0,
	\end{equation*}
	which implies the continuity equation for the density $\rho^\e$:
	\[\partial_t \rho^\e +\nabla\cdot (\rho^\e u^\e) = 0,\quad u^\e:=\nabla S^\e+A^\e.\]
	We note that the definition of $u^\e$ implies
	\begin{equation}\label{uA}
		\partial_i u^\e_j-\partial_j u^\e_i = \partial_i(\partial_j S^\e+A^\e_j)-\partial_j(\partial_i S^\e+A^\e_i) =  \partial_i A^\e_j -\partial_j A^\e_i,\quad\mbox{i.e.},\quad \nabla\times u^\e = \nabla\times A^\e.
	\end{equation}
	On the other hand, we note that the right-hand sides of \eqref{additional_2} become
	\begin{align*}
	-&\e\Im\left(\overline{\psi^\e}D^\e_2\psi^\e\right) =-\e\Im\left(\sqrt{\rho^\e}\left(\pa_2\sqrt{\rho^\e}+\frac{\i}{\e}\sqrt{\rho^\e}\pa_2S^\e+\frac{\i}{\e}A_2^\e\sqrt{\rho^\e}\right)\right)=-\rho^\e\pa_2S^\e-\rho^\e A_2^\e=-\rho^\e u^\e_2,\\
	&\e\Im\left(\overline{\psi^\e}D^\e_1\psi^\e\right)=\rho^\e\pa_1S^\e+\rho^\e A_1^\e=\rho^\e u_1^\e,
	\end{align*}
    and the right-hand side of \eqref{CSS-e}$_2$ becomes
	\begin{align*}
	\e&\Im(Q_{12}(\overline{\psi^\e},\psi^\e))+\pa_1(A_2^\e|\psi^\e|^2)-\pa_2(A_1^\e|\psi^\e|^2)\\
	&=\e\Im\left(\pa_1\overline{\psi^\e}\pa_2\psi^\e-\pa_2\overline{\psi^\e}\pa_1\psi^\e\right)+\pa_1(A_2^\e\rho^\e)-\pa_2(A_1^\e\rho^\e)\\
	&=2\sqrt{\rho^\e}(\pa_1\sqrt{\rho^\e}\pa_2S^\e-\pa_2\sqrt{\rho^\e}\pa_1S^\e)+(\pa_1A_2^\e-\pa_2A_1^\e)\rho^\e +A_2^\e\pa_1\rho^\e-A_1^\e\pa_2\rho^\e\\
	&=\pa_1\rho^\e(\pa_2S^\e+A_2^\e)-\pa_2\rho^\e(\pa_1S^\e+A_1^\e)+(\nabla\times A^\e)\rho^\e\\
	&=(\pa_1\rho^\e) u_2^\e-(\pa_2\rho^\e) u_1^\e+(\nabla\times A^\e)\rho^\e = \nabla\rho^\e \times u^\e+(\nabla\times u^\e) \rho^\e,
	\end{align*}
	where we used \eqref{uA} in the last identity. Therefore, the Chern-Simons parts \eqref{CSS-e}$_{2,3}$ become
	\begin{align*}
	\begin{aligned}
     &\Delta A^\e_0 =\nabla\times(\rho^\e u^\e),\quad
	\Delta A^{\e}=-(\nabla \rho^{\e})^{\perp},
	\end{aligned}
	\end{align*}
	where $(u^\e)^\perp :=(-u_2^\e,u_1^\e)$.
We also note that the additional relation \eqref{additional_2} can be written as
\begin{align}\label{additional-e}
		\pa_t A^\e =\nabla A_0^\e +\rho^\e (u^\e)^\perp.
\end{align}
 Finally, to identify the momentum equation, we consider the real part of the Schr\"odinger equation \eqref{CSS-e}$_1$:
	\begin{equation}\label{eq-S}
		\partial_t S^\e+A^\e_0-\frac{\e^2}{2}\frac{\Delta \sqrt{\rho^\e}}{\sqrt{\rho^\e}}+\frac{1}{2}|\nabla S^\e+A^\e|^2+V'(\rho^\e)=0,
	\end{equation}
	and then take a gradient to \eqref{eq-S} to obtain
	
	\begin{equation}\label{eq-nablaS}
	\partial_t (\nabla S^\e) + \frac{1}{2}\nabla|\nabla S^\e+A^\e|^2+\nabla A^\e_0 + V''(\rho^\e)\nabla\rho^\e = \frac{\e^2}{2}\nabla\left(\frac{\Delta \sqrt{\rho^\e}}{\sqrt{\rho^\e}}\right).
	\end{equation}
	We note that for any $u\in\bbr^2$, the following identity holds:
	\[\frac{1}{2}(\nabla|u|^2)_i = \sum_{j=1}^2 u_j\partial_i u_j = \sum_{j=1}^2 u_j\partial_ju_i +u_j(\partial_iu_j-\partial_ju_i)= (u\cdot \nabla)u_i +\sum_{j=1}^2u_j(\partial_iu_j-\partial_ju_i).\]
	Therefore, we modify the second term in \eqref{eq-nablaS} as
	\begin{align*}
	\frac{1}{2}(\nabla|\nabla S^\e+A^\e|^2)_i=\frac{1}{2}(\nabla|u^\e|^2)_i = (u^\e\cdot \nabla)u^\e_i+\sum_{j=1}^2u^\e_j(\partial_iu^\e_j-\partial_j u^\e_i),
	\end{align*}
	and again use \eqref{uA} and \eqref{additional} to obtain
	\[u^\e_2(\partial_1 u^\e_2-\partial_2 u^\e_1) = u^\e_2(\partial_1A^\e_2-\partial_2 A^\e_1) = -\rho^\e u^\e_2,\quad u^\e_1(\partial_2 u^\e_1-\partial_1 u^\e_2) = u^\e_1(\partial_2A^\e_1-\partial_1 A^\e_2)= \rho^\e u^\e_1.\]
	Hence, the equation \eqref{eq-nablaS} becomes
	\begin{equation}\label{eq-nablaS-2}
	\partial_t(\nabla S^\e) +(u^\e\cdot \nabla) u^\e +\rho^\e (u^\e)^\perp+\nabla A^\e_0 +\frac{\nabla p(\rho^\e)}{\rho^\e} = \frac{\e^2}{2}\nabla\left(\frac{\Delta \sqrt{\rho^\e}}{\sqrt{\rho^\e}}\right),
	\end{equation}
	where $p(\rho):=\rho V'(\rho) -V(\rho)$ is a pressure. Thus, combining the estimates \eqref{eq-nablaS-2} for $\nabla S^\e$ and \eqref{additional-e} for $A^\e_0$, we finally derive the governing equation for $u^\e$: 
	\[\partial_t u^\e +(u^\e\cdot \nabla) u^\e +\frac{\nabla p(\rho^\e)}{\rho^\e} =  \frac{\e^2}{2}\nabla\left(\frac{\Delta \sqrt{\rho^\e}}{\sqrt{\rho^\e}}\right).\]
	In conclusion, the quantities $(\rho^\e,\rho^\e u^\e,A_0^\e,A^\e)$ satisfy \eqref{CSS-hydro-e}:
	\begin{align*}
	&\partial_t \rho^\e +\nabla\cdot(\rho^\e u^\e) = 0,\\
	&\partial_t (\rho^\e u^\e) +\nabla\cdot(\rho^\e u^\e\otimes u^\e) +\nabla p(\rho^\e) =\rho^\e\frac{\e^2}{2}\nabla\left(\frac{\Delta \sqrt{\rho^\e}}{\sqrt{\rho^\e}}\right),\\
	&\Delta A^\e_0 = \nabla	\times(\rho^{\e} u^{\e}),\quad
	\Delta A^{\e}=-(\nabla \rho^{\e})^{\perp}.
	\end{align*}
	\begin{remark}
		Since the total charge $\int_{\Omega}|\psi^\e|^2\,\d x$ is conserved by Proposition \ref{P2.1}, we also obtain that the total density $\int_{\Omega}\rho^\e\,\d x$ is also conserved. Hence, we normalize it to be 1. 
	\end{remark}
	\begin{remark}
	To guarantee the well-posedness of the Poisson equations, one need to impose the boundary condition $A_\mu^\e(t,x) \to 0$ as $|x|\to\infty$ for the case of whole domain, and the normalization condition $\int_{\bbt^2} A_\mu^\e\,\d x=0$ for the case of periodic domain. 
	\end{remark}

	\subsubsection{Derivation of \eqref{CSS-hydro-e-2}} \label{sec:2.2.2}
	Since the overall procedure for deriving \eqref{CSS-hydro-e-2} is the same as the derivation of \eqref{CSS-hydro-e}, we just briefly mention the several key equations.
	Considering the Madelung transformation \eqref{madelung} again, the imaginary part of the Schr\"odinger equation \eqref{CSS-e-incomp}$_1$ becomes
	\[\e^{1+\alpha}\pa_t \sqrt{\rho^\e} +\e\nabla\sqrt{\rho^\e}\cdot \nabla S^\e +\frac{\sqrt{\rho^\e}}{2}\e\Delta S^\e +\e^{1-\alpha} A^\e\cdot \nabla\sqrt{\rho^\e}=0,\]
	which is equivalent to
	\[\pa_t\rho^\e +\nabla\cdot(\rho^\e u^\e) =0,\quad u^\e:=\e^{-\alpha}\nabla S^\e +\e^{-2\alpha}A^\e.\]
	On the other hand, the real part of \eqref{CSS-e-incomp}$_1$ becomes
	\[\e^{\alpha}\pa_t S^\e+\e^\alpha A_0^\e +\frac{\e^{2\alpha}}{2}|\e^{-\alpha}\nabla S^\e+\e^{-2\alpha}A^\e|^2+((\rho^\e)^{\gamma-1}-(\rho^\e)^{\frac{\gamma}{2}-1})=\frac{\e^2}{2}\frac{\Delta \sqrt{\rho^\e}}{\sqrt{\rho^\e}},\]
	and after taking the gradient, we have
	\begin{align}
	\begin{aligned}\label{B-7}
	\pa_t(\e^{-\alpha}\nabla S^\e)&+\e^{-\alpha}\nabla A_0^\e +(u^\e\cdot\nabla)u^\e+\e^{-\alpha}\rho^\e (u^\e)^\perp\\
	&+\e^{-2\alpha}\left((\gamma-1)(\rho^\e)^{\gamma-2}-\left(\frac{\gamma}{2}-1\right)(\rho^\e)^{\frac{\gamma}{2}-2}\right)\nabla\rho^\e=\frac{\e^{2-2\alpha}}{2}\nabla \left(\frac{\Delta \sqrt{\rho^\e}}{\sqrt{\rho^\e}}\right).
	\end{aligned}
	\end{align}
	The Chern-Simons parts \eqref{CSS-e-incomp}$_{2,3}$ and additional propagating condition become
	\begin{align}
		\begin{aligned}\label{B-8}
	 	&\Delta A_0^\e = \nabla \times (\rho^\e u^\e),\quad \Delta A^{\e}=-\e^{\alpha}(\nabla \rho^{\e})^{\perp},\\
			&\e^{-\alpha}\pa_t A^\e =\nabla A_0^\e +\rho^\e (u^\e)^\perp.\\
		\end{aligned}
	\end{align}
	Therefore, we combine \eqref{B-7} and \eqref{B-8}$_2$ to conclude that the momentum equation becomes 
	\[\pa_t (\rho^\e u^\e) +\nabla\cdot(\rho^\e u^\e\otimes u^\e )+\e^{-2\alpha}\nabla\left(\frac{\gamma-1}{\gamma}(\rho^\e)^\gamma-\frac{\gamma-2}{\gamma}(\rho^\e)^{\frac{\gamma}{2}}\right)=\frac{\e^{2-2\alpha}\rho^\e}{2}\nabla\left(\frac{\Delta \sqrt{\rho}}{\sqrt{\rho}}\right).\]
	Therefore, the quantities $(\rho^\e,\rho^\e u^\e,A_0^\e,A^\e)$ satisfy \eqref{CSS-hydro-e-2}:
	\begin{align*}
		&\partial_t \rho^\e +\nabla\cdot(\rho^\e u^\e) = 0,\\
		&\partial_t (\rho^\e u^\e) +\nabla\cdot(\rho^\e u^\e\otimes u^\e) +\e^{-2\alpha}\nabla\left(\frac{\gamma-1}{\gamma}(\rho^\e)^\gamma-\frac{\gamma-2}{\gamma}(\rho^\e)^{\frac{\gamma}{2}}\right) =\rho^\e\frac{\e^{2-2\alpha}}{2}\nabla\left(\frac{\Delta \sqrt{\rho^\e}}{\sqrt{\rho^\e}}\right),\\
		&\Delta A^\e_0 = \nabla	\times(\rho^{\e} u^{\e}),\quad \Delta A^{\e}	=-\e^{\alpha}(\nabla \rho^{\e})^{\perp}.
	\end{align*}

	\subsection{Well-posedness of the system}
	In this part, we briefly review the well-posedness part of the systems \eqref{CSS-e}, \eqref{CSE} and \eqref{CSE-incomp}. The well-posedness of the nonlinear Schr\"odinger-Chern-Simons system has been studied in numerous previous literature. For the focusing potential $V(x)=-cx^2$, the local well-posedness in $H^2$ space and the global existence in $H^1$ with sufficiently small initial data were first studied in \cite{BBS95}, and the local existence and unconditional uniqueness of the $H^1$ initial without smallness assumption on the initial data were proved in \cite{H13}. The low-regularity well-posedness of the system \eqref{CSS-e} is also investigated in \cite{LS16, LST14}. Especially, for the defocusing potential (the case to be discussed in this paper), the existence of unique global  solution $\psi\in C([0,\infty);H^2)$ was recently investigated for arbitrary initial data $\psi_{\textup{in}}\in H^2$ in \cite{L18}. On the other hand, the local-in-time smooth solution ($H^s$ with sufficiently large $s$) to the compressible Euler system \eqref{CSE}$_{1,2}$ and the incompressible Euler system \eqref{CSE-incomp}$_1$ was already well-studied in previous literature, e.g., \cite{P96,M84,MB02}. Therefore, in this paper, we consider the local-in-time smooth solution $(\rho, u)\in C([0,T_*);H^s)$ ($u\in C([0,T_*);H^s)$ for the incompressible fluid) with $s>3$, where $T_*$ is the maximal existence time for the Euler-type equations.

	\subsection{Main theorems}
	We now present the exact statements of the hydrodynamic limit of the scaled Schr\"odinger-Chern-Simons equations, which are the main theorems of the paper. We first consider the hydrodynamic limit from \eqref{CSS-hydro-e} toward \eqref{CSE}. To this end, the following {\it well-prepared initial data} condition is required:
	
	\begin{align}
		\begin{aligned}\label{well-prepared}
	\int_{\Omega}\frac{\rho_{\textup{in}}^\e|u_{\textup{in}}^\e-u_{\textup{in}}|^2}{2}\,\d x+\int_{\Omega}\frac{p(\rho_{\textup{in}}^\e|\rho_{\textup{in}})}{\gamma-1}\,\d x&+\frac{\e^2}{2}\int_{\Omega}|\nabla\sqrt{\rho_{\textup{in}}^\e}|^2\,\d x=\mathcal{O}(\e^\lambda),\quad \lambda>0,
		\end{aligned}
	\end{align}
	where $p(n|\rho):=\frac{\gamma-1}{\gamma}(n^\gamma-\rho^\gamma-\gamma\rho^{\gamma-1}(n-\rho))$. Note that the condition \eqref{well-prepared} holds if
	\[\|u^\e_{\textup{in}}-u_{\textup{in}}\|_{L^\infty} = \mathcal{O}(\e^{\lambda/2}),\quad \|p(\rho^\e_{\textup{in}}|\rho_{\textup{in}})\|_{L^1}=\mathcal{O}(\e^\lambda),\quad \|\nabla\sqrt{\rho_{\textup{in}}^\e}\|_{L^2}=\mathcal{O}(\e^{\lambda/2-1}).\]
	
	\begin{theorem}\label{thm:main-1}
		Suppose $\gamma\ge 2$ if $\Omega=\bbr^2$ or $\gamma> 1$ if $\Omega=\bbt^2$. Let $(\psi^\e,A^\e_0,A^\e)$ be the unique global solution to the Schr\"odinger-Chern-Simons equations \eqref{CSS-e}, subject to the initial data $\psi^\e_{\textup{in}}\in H^2(\Omega)$ and $A^\e_{\textup{in}}$ satisfying the compatibility condition \eqref{compatibility}. Suppose that the non-vacuum condition uniformly holds, i.e., $|\psi^\e(t,x)|>0$ for $(t,x)\in [0,T_*)\times \Omega$ and define the hydrodynamic quantities:
		\[\rho^\e:=|\psi^\e|^2 = \psi^\e\overline{\psi^\e},\quad u^\e := \nabla S^\e +A^\e = \frac{\i\e}{2|\psi^\e|^2}(\psi^\e\nabla \overline{\psi^\e}-\overline{\psi^\e}\nabla\psi^\e)+A^\e.\]
		Moreover, let $(\rho,u,A_0,A)$ be the unique local-in-time smooth solution to the Compressible Euler-Chern-Simons equations \eqref{CSE} for $0\le t< T_*$, subject to the initial data $(\rho_{\textup{in}},u_{\textup{in}}, A_{\textup{in}})$ satisfying the compatibility condition \eqref{compatibility-CE}. Finally, suppose that the initial data satisfy the well-prepared condition \eqref{well-prepared}. Then, for any $0\le t< T_*$, we have
			\begin{align*}
				&\rho^\e(t,\cdot)\to \rho(t,\cdot)\quad \mbox{\textup{in}}\quad L^\gamma(\Omega),\\
				&(\rho^\e u^\e)(t,\cdot) \to (\rho u)(t,\cdot)\quad \mbox{\textup{in}}\quad L^{\frac{2\gamma}{\gamma+1}}(\Omega),\\
				&(\sqrt{\rho^\e} u^\e)(t,\cdot)\to (\sqrt{\rho}u)(t,\cdot),\quad \mbox{\textup{in}}\quad L^2(\Omega),\\
				&A^\e_0 \to A_0\quad\mbox{\textup{in}}\quad L^{2\gamma}(\Omega),\quad \nabla A^\e_0\to \nabla A_0\quad\mbox{\textup{in}}\quad L^{\frac{2\gamma}{\gamma+1}}(\Omega),\\
				&A^\e\to A\quad \mbox{\textup{in}}\quad L^{2\gamma}(\Omega),\quad \nabla A^\e \to\nabla A \quad\mbox{\textup{in}}\quad L^\gamma(\Omega),
			\end{align*}
			as the scaling parameter $\e$ tends to 0.
	\end{theorem}

	On the other hand, for the incompressible hydrodynamic limit, we need the following different well-prepared initial data condition
	\begin{align}
		\begin{aligned}\label{well-prepared-2}
		\int_{\Omega}\frac{\rho_{\textup{in}}^\e|u_{\textup{in}}^\e-u_{\textup{in}}|^2}{2}\,\d x&+\frac{\e^{-2\alpha}}{\gamma}\int_{\Omega}((\rho_{\textup{in}}^\e)^{\frac{\gamma}{2}}-1)^2\,\d x
		+\frac{\e^{2-2\alpha}}{2}\int_{\Omega}|\nabla\sqrt{\rho_{\textup{in}}^\e}|^2\,\d x=\mathcal{O}(\e^\lambda),\,\,\lambda>0.
		\end{aligned}
	\end{align}
	\begin{theorem}\label{thm:main-2}
	Suppose $\gamma\ge 2$ and $\Omega=\bbt^2$. Let $(\psi^\e,A^\e_0,A^\e)$ be the unique global solution to the Schr\"odinger-Chern-Simons equations \eqref{CSS-e-incomp}, subject to the initial data $\psi^\e_{\textup{in}}\in H^2(\Omega)$ and $A^\e_{\textup{in}}$ satisfying the compatibility condition \eqref{compatibility}. Suppose that the non-vacuum condition uniformly holds, i.e., $|\psi^\e(t,x)|>0$ for $(t,x)\in[0,T_*)\times\Omega$ and define the hydrodynamic quantities:
	\[\rho^\e:=|\psi^\e|^2 = \psi^\e\overline{\psi^\e},\quad u^\e := \e^{-\alpha}\nabla S^\e +\e^{-2\alpha}A^\e = \frac{\i\e^{1-\alpha}}{2|\psi^\e|^2}(\psi^\e\nabla \overline{\psi^\e}-\overline{\psi^\e}\nabla\psi^\e)+\e^{-2\alpha}A^\e.\]
	Moreover, let $(u,A_0)$ be the unique local-in-time smooth solution to the incompressible Euler-Poisson equations \eqref{CSE-incomp} for $0\le t< T_*$, subject to the initial data $u_{\textup{in}}$. Finally, suppose that the initial data satisfy the well-prepared condition \eqref{well-prepared-2}. Then, for any $0\le t< T_*$, we have
	\begin{align*}
		&\rho^\e(t,\cdot)\to 1\quad \mbox{\textup{in}}\quad L^\gamma(\Omega),\\
		&(\rho^\e u^\e)(t,\cdot) \to u(t,\cdot)\quad \mbox{\textup{in}}\quad L^{\frac{2\gamma}{\gamma+1}}(\Omega),\\
		&(\sqrt{\rho^\e} u^\e)(t,\cdot)\to u(t,\cdot),\quad \mbox{\textup{in}}\quad L^2(\Omega),\\
		&A^\e_0 \to A_0\quad\mbox{\textup{in}}\quad L^{2\gamma}(\Omega),\quad \nabla A^\e_0\to \nabla A_0\quad\mbox{\textup{in}}\quad L^{\frac{2\gamma}{\gamma+1}}(\Omega),\\
		&A^\e\to 0\quad \mbox{\textup{in}}\quad L^{2\gamma}(\Omega),\quad \nabla A^\e \to 0 \quad\mbox{\textup{in}}\quad L^\gamma(\Omega),
	\end{align*}
	as the scaling parameter $\e$ tends to 0. 
\end{theorem}

	In the following sections, we present the detailed proof for Theorem \ref{thm:main-1} and Theorem \ref{thm:main-2}. Precisely, in Section \ref{sec:3} and Section \ref{sec:4}, we focus on proving Theorem \ref{thm:main-1}. The main idea is to estimate the modulated energy correspond to the SCS equations, which can control the difference between $(\rho^\e,u^\e)$ and $(\rho,u)$. To estimate the modulated energy, we review the relative entropy framework, which is frequently used in the hydrodynamic limit of the Vlasov-type equations to the Euler-type equations. We reveal that the modulated energy for the SCS system is the quantum-perturbation of the classical relative entropy of the Euler-type equations, and by utilizing the celebrated theory on the relative entropy method, we derive a simple estimate on the modulated energy. Once we have the estimate on the modulated energy, we can obtain the convergence of the hydrodynamic quantities $(\rho^\e,u^\e)$, from which the convergence of the gauge fields $A_\mu^\e$ is derived using the estimates in the elliptic PDE theory. In Section 5, we focus on proving Theorem \ref{thm:main-2}. Unfortunately, the relative entropy theory cannot be applied to the system \eqref{CSE-incomp}. Instead, we directly estimate the modulated energy of the SCS system \eqref{CSS-e-incomp} to obtain the desired convergence.

	\section{Relative entropy and the modulated energy}\label{sec:3}
	\setcounter{equation}{0}
	In this section, we introduce the main technical tools for the hydrodynamic limit of the SCS system towards the compressible ECS system. Precisely, we first review the relative entropy method, which has been used to study the hydrodynamic limit of particle \cite{GLT09,L04,Y91} and kinetic systems \cite{BGL00, BGP00, GJV04, GS04, KV15, MV08, S03, S09}. The main idea of the relative entropy method is to estimate the functional called \textit{relative entropy}, or sometimes called the \textit{modulated energy}, which measures the difference between the solution to the scaled system and the solution to the asymptotic system. This quantity is shown to vanish as the scaling parameter $\e$ tends to 0, guaranteeing the desired convergence of the solution toward the solution to the target equation. On the other hand, the modulated energy estimate was also used to obtaining the hydrodynamic limit for the Schr\"odinger-type equations \cite{LL05,LL08,LW12,LZ06,P02}. Therefore, we first review the relative entropy method for the classical conservation laws and provide the close relation between relative entropy of the compressible Euler equations and the modulated energy of the SCS system, which will be crucially used in the proof of the main theorem.
	
	\subsection{Review on the relative entropy method}
	We consider the following general system of conservation laws:
	\begin{equation}\label{conservation}
	\partial_t U_i +\sum_{k=1}^d \partial_{k}\bm{A}_{ik}(U) = 0,
	\end{equation}
	where $U\in \bbr^{m}$ is the state and $\bm{A}\in\bbr^{m\times d}$ is the flux. Indeed, the compressible Euler equations \eqref{CSE}$_{1,2}$ can be written in the form of \eqref{conservation} with $m=3$, $d=2$, $U=(\rho,P):=(\rho,\rho u)$ and
	\[\bm{A}(U):=\frac{1}{\rho}\begin{pmatrix}\rho P_1 &\rho P_2\\ P_1^2 +\frac{\gamma-1}{\gamma}\rho^{\gamma+1} & P_1P_2\\ P_2P_1& P_2^2+\frac{\gamma-1}{\gamma}\rho^{\gamma+1}\end{pmatrix}=\begin{pmatrix}
	\rho u^\top\\ \rho u\otimes u +\frac{\gamma-1}{\gamma}\rho^\gamma I_2
	\end{pmatrix}.\]
	For the compressible Euler equation, one usually consider the following natural (mathematical) entropy $\eta(U)$ defined as
	\[\eta(U):=\frac{|P|^2}{2\rho}+\frac{\rho^\gamma}{\gamma}=\frac{\rho|u|^2}{2}+\frac{\rho^\gamma}{\gamma}.\]
	For any function $V = (n,nv)$, we define the relative entropy $\eta(V|U)$ and the relative flux $\bm{A}(V|U)$ as follows:
	\begin{align*}
	&\eta(V|U) := \eta(V)-\eta(U) -D\eta(U)\cdot(V-U),\\
	&\bm{A}(V|U) := \bm{A}(V) - \bm{A}(U) -D\bm{A}(U)\cdot (V-U),
	\end{align*}
	where $D\bm{A}(U)\cdot (V-U)$ is a $3\times 2$ matrix defined as
	\[[D\bm{A}(U)\cdot(V-U)]_{ij}:=\sum_{k=1}^{3}\partial_{U_k}\bm{A}_{ij}(U)(V_k-U_k).\]
	By direct computation, one can easily compute $D\eta$ as
	\[D\eta = \begin{pmatrix}D_\rho \eta\\ D_P\eta\end{pmatrix}=\begin{pmatrix} -\frac{|P|^2}{2\rho^2}+\rho^{\gamma-1}\\ \frac{P}{\rho}\end{pmatrix}=\begin{pmatrix}
	-\frac{|u|^2}{2} +\rho^{\gamma-1}\\ u
	\end{pmatrix}\]
	and therefore, the relative entropy becomes
	\begin{align}
	\begin{aligned}\label{rel-ent-euler}
	\eta(V|U) &= \frac{n|v|^2}{2}+\frac{n^\gamma}{\gamma}-\frac{\rho|u|^2}{2}-\frac{\rho^{\gamma}}{\gamma}-\left(-\frac{|u|^2}{2}+\rho^{\gamma-1}\right)(n-\rho)-u\cdot(nv-\rho u)\\
	&=\frac{n|u-v|^2}{2} + \frac{1}{\gamma}(n^\gamma-\rho^\gamma-\gamma\rho^{\gamma-1}(n-\rho))=\frac{n|u-v|^2}{2}+\frac{p(n|\rho)}{\gamma-1}.
	\end{aligned}
	\end{align}
	Moreover, using the definition of $D\bm{A}(U)\cdot(V-U)$, it is straightforward to observe that
	\[D\bm{A}(U)\cdot(V-U) = \begin{pmatrix}
	(nv-\rho u)^\top\\
	-(n-\rho)(u\otimes u)+u\otimes (nv-\rho u)+(nv-\rho u)\otimes u +(\gamma-1) \rho^{\gamma-1}(n-\rho)I_2
	\end{pmatrix}.\] 
	Therefore, the relative flux $\bm{A}(V|U)$ is computed as
	\begin{align*}
	\bm{A}(V|U) = \begin{pmatrix}
	0\\
	n(v-u)\otimes(v-u)+p(n|\rho)I_2
	\end{pmatrix}.
	\end{align*}
	The hydrodynamic limit towards the Euler-type system is based on the following proposition, which is the cornerstone of the relative entropy method.
	
	\begin{proposition}\cite{D79}\label{prop:rel-ent}
		Let $U$ be a smooth solution to the conservation law \eqref{conservation}, and let $V$ be any function. Then, the following estimate on the relative entropy holds:
		\begin{align*}
		\frac{\d}{\d t}\int_{\bbr^d} \eta(V|U)\,\d x&=\frac{\d}{\d t}\int_{\bbr^d}\eta(V)\,\d x - \int_{\bbr^d}\nabla_x(D\eta(U)):\bm{A}(V|U)\,\d x\\
		&\quad-\int_{\bbr^d}D\eta(U)\cdot\left(\pa_t V+\nabla_x\cdot \bm{A}(V)\right)\,\d x. 
		\end{align*}
	\end{proposition}
	
	\subsection{Relation between relative entropy and the modulated energy}
	In this part, we first show that the energy function $\mathcal{E}^\e$ can be represented in terms of the hydrodynamic variables $(\rho^\e,\rho^\e u^\e)$. Then, we introduce the modulated energy, which is used to show the hydrodynamic limit of the Schr\"odinger-type equations, and show that it can be understood as a perturbation of the relative entropy \eqref{rel-ent-euler} of the compressible Euler equation. First, the energy function can be expanded as
	\begin{align*}
	\mathcal{E}^\e(t,x) = \frac{1}{2}|\e D^\e\psi^\e|^2 +V(|\psi^\e|^2) =\frac{1}{2}(\e\nabla \psi^\e +\i A^\e\psi^\e)\cdot(\e\nabla\overline{\psi^\e}-\i A^\e\overline{\psi^\e})+V(\rho^\e).
	\end{align*}
	We also note that
	\begin{align*}
	(\e\nabla \psi^\e +\i A^\e\psi^\e)\cdot(\e\nabla\overline{\psi^\e}-\i A^\e\overline{\psi^\e}) = \e^2\nabla\psi^\e\cdot\nabla\overline{\psi^\e}+\i\e A^\e\cdot(\psi^\e\nabla \overline{\psi^\e}-\overline{\psi^\e}\nabla{\psi^\e})+|A^\e|^2|\psi^\e|^2.
	\end{align*}
	On the other hand, it follows from the definition $u^\e=\nabla S^\e+A^\e$ that 
	\[\rho^\e|u^\e|^2 = \rho^\e\left|\frac{\i\e}{2\rho^\e}(\psi^\e\nabla\overline{\psi^\e}-\overline{\psi^\e}\nabla\psi^\e)+A^\e\right|^2 = \frac{\e^2}{4\rho^\e}|\psi^\e\nabla\overline{\psi^\e}-\overline{\psi^\e}\nabla\psi^\e|^2+\i\e(\psi^\e\nabla\overline{\psi^\e}-\overline{\psi^\e}\nabla\psi^\e)\cdot A^\e+\rho^\e|A^\e|^2.\]
	Therefore, one has
	\begin{align*}
	(\e\nabla &\psi^\e +\i A^\e\psi^\e)\cdot(\e\nabla\overline{\psi^\e}-\i A^\e\overline{\psi^\e}) \\
	&= \rho^\e|u^\e|^2+\e^2\nabla\psi^\e\cdot\nabla \overline{\psi^\e}-\frac{\e^2}{4\rho^\e}|\psi^\e\nabla\overline{\psi^\e}-\overline{\psi^\e}\nabla\psi^\e|^2\\
	&=\rho^\e|u^\e|^2+\frac{\e^2}{4\rho^\e}|\psi^\e\nabla\overline{\psi^\e}+\overline{\psi^\e}\nabla\psi^\e|^2 = \rho^\e|u^\e|^2 +\frac{\e^2}{4\rho^\e}|\nabla\rho^\e|^2 = \rho^\e|u^\e|^2+\e^2|\nabla\sqrt{\rho^\e}|^2.
	\end{align*}
	We substitute the above estimate to the energy function to write the energy function in terms of the hydrodynamic quantities:
	\begin{align}
	\begin{aligned}\label{E-hydro}
	\mathcal{E}^\e&=\frac{1}{2}\rho^\e|u^\e|^2+V(\rho^\e)+\frac{1}{2}\e^2|\nabla\sqrt{\rho^\e}|^2= \frac{1}{2}\rho^\e |u^\e|^2 +\frac{(\rho^\e)^{\gamma}}{\gamma}+\frac{\e^2}{2}|\nabla\sqrt{\rho^\e}|^2\\
	&=\eta(U^\e)+\frac{\e^2}{2}|\nabla \sqrt{\rho^\e}|^2,
	\end{aligned}
	\end{align}
	where $U^\e:=(\rho^\e,\rho^\e u^\e)$. Therefore, one can observe that the energy function of the SCS equations can be understood as a modified entropy with the additional term $\frac{\e^2}{2}|\nabla \sqrt{\rho^\e}|^2$. We interpret this additional term as a quantum correction.\\
	
	\noindent On the other hand, considering the modulated energy as in the Schr\"odinger-type equations \cite{LL08,LW12,LZ06}, the natural modulated energy for the SCS equations is 
	\begin{align*}
	\mathcal{H}^\e(t)&:=\int_{\Omega}\frac{1}{2}|(\e D^\e-\i u)\psi^\e|^2+\frac{p(\rho^\e|\rho)}{\gamma-1}\,\d x\\
	&=\int_{\Omega}\frac{1}{2}|(\e D^\e-\i u)\psi^\e|^2+\frac{(\rho^\e)^\gamma-\rho^\gamma-\gamma\rho^{\gamma-1}(\rho^\e-\rho)}{\gamma}\,\d x.
	\end{align*}
	In what follows, we show that the modulated energy can also be represented in terms of the hydrodynamic variables. We first expand the modulated energy as
	\begin{align*}
	\mathcal{H}^\e&=\int_{\Omega}\frac{\e^2}{2}|D^\e\psi^\e|^2+\frac{p(\rho^\e|\rho)}{\gamma-1}\,\d x +\frac{\e\i}{2}\int_{\Omega}(D^\e\psi^\e\cdot u \overline{\psi^\e}-\overline{D^\e\psi^\e}\cdot  u\psi^\e) \d x+\frac{1}{2}\int_{\Omega}|u|^2|\psi^\e|^2\,\d x\\
	&=\int_{\Omega}\mathcal{E}^\e(t,x)\,\d x+\int_{\Omega}\rho^{\gamma-1}\left(\frac{\gamma-1}{\gamma}\rho-\rho^{\e}\right)\,\d x\\
	&\quad +\frac{\e\i}{2}\int_{\Omega}\left(\overline{\psi^\e}\nabla\psi^\e +\frac{\i}{\e}A^\e|\psi^\e|^2-\psi^\e\nabla\overline{\psi^\e}+\frac{\i}{\e}A^\e|\psi^\e|^2\right)\cdot u +\frac{1}{2}\int_{\Omega}\rho^\e|u|^2\,\d x.
	\end{align*}
	However, since
	\begin{align*}
	\frac{\e\i}{2}\left(\overline{\psi^\e}\nabla\psi^\e +\frac{\i}{\e}A^\e|\psi^\e|^2-\psi^\e\nabla\overline{\psi^\e}+\frac{\i}{\e}A^\e|\psi^\e|^2\right) = \frac{\e\i}{2}\left(\overline{\psi^\e}\nabla\psi^\e -\psi^\e\nabla\overline{\psi^\e}\right)-A^\e|\psi^\e|^2=-\rho^\e u^\e,
	\end{align*}
	the modulated energy becomes
	\[\mathcal{H}^\e = \int_{\Omega}\mathcal{E}^\e(t,x)\,\d x+\int_{\Omega}\rho^{\gamma-1}\left(\frac{\gamma-1}{\gamma}\rho-\rho^\e\right)\,\d x-\int_{\Omega}\rho^\e u^\e\cdot u\,\d x+\frac{1}{2}\int_{\Omega}\rho^\e|u|^2\,\d x.\]
	Considering the hydrodynamic formulation \eqref{E-hydro} for $\mathcal{E}^\e$, the modulated energy is related to classical relative entropy with the quantum correction term:
	\begin{align}
		\begin{aligned}\label{ME-RE}
			\mathcal{H}^\e&=\int_{\Omega}\frac{\rho^\e|u^\e-u|^2}{2}+\frac{p(\rho^\e|\rho)}{\gamma-1}\,\d x+\frac{\e^2}{2}\int_{\Omega}|\nabla\sqrt{\rho^\e}|^2\,\d x\\
			&=\int_{\Omega}\eta(U^\e|U)\,\d x+\frac{\e^2}{2}\int_{\Omega}|\nabla \sqrt{\rho^\e}|^2\,\d x.
		\end{aligned}
	\end{align}
	Therefore, the relation \eqref{ME-RE} reveals that the theory of relative entropy method developed in the classical hydrodynamic limit problem can also be applied to the hydrodynamic limit problem of the Schr\"odinger-type equations. In the following section, we present the estimate of the modulated energy $\mathcal{H}^\e$, with the help of the classical relative entropy estimate.
	
	\section{Hydrodynamic limit to the compressible Euler-Chern-Simons system}\label{sec:4}
	\setcounter{equation}{0}
	In this section, we provide the proof of Theorem \ref{thm:main-1}. We first estimate the modulated energy function to show that it disappears as $\e$ converges to 0. Then, using the decay estimate of the modulated energy, we show that the desired convergences are attained. 
	
	\subsection{Modulated energy estimate}
	We start by deriving the estimate of the modulated energy. The main idea is to use the relation between the modulated energy and the relative entropy, and the well-known classical estimate on the relative entropy in Proposition \ref{prop:rel-ent}.
	
	\begin{proposition}\label{P4.1}
		Let $(\psi^\e,A^\e_0,A^\e)$ be the unique global solution to the SCS equations \eqref{CSS-e} subject to the initial data $\psi^\e_{\textup{in}}\in H^2(\Omega)$ and  $A^\e_{\textup{in}}$ satisfying \eqref{compatibility} and let $(\rho,u)$ be the unique local-in-time smooth solution to the compressible Euler equations \eqref{CSE}$_{1,2}$ subject to the initial data $(\rho_{\textup{in}},u_{\textup{in}})$. Suppose that the initial data satisfy the well-prepared condition \eqref{well-prepared}. Then,
		\begin{equation}\label{est-modulated}
		\mathcal{H}^\e(t)\le C\e^{\min\{\lambda,2\}},\quad 0\le t<T_*.
		\end{equation}
	\end{proposition}
	\begin{proof}
		We use the relation \eqref{ME-RE} between the modulated energy $\mathcal{H}^\e$ and the relative entropy $\eta(U^\e|U)$, Proposition \ref{prop:rel-ent}, and the representation \eqref{E-hydro} of the energy to derive
		\begin{align*}
		\frac{\d \mathcal{H}^\e}{\d t} &= \frac{\d}{\d t}\int_{\Omega}\eta(U^\e|U)\,\d x+\frac{\e^2}{2}\frac{\d}{\d t}\int_{\Omega}|\nabla\sqrt{\rho^\e}|^2\,\d x\\
		&=\frac{\d}{\d t}\int_{\Omega}\eta(U^\e)\,\d x -\int_{\Omega}\nabla_x(D\eta(U)):\bm{A}(U^\e|U)\,\d x-\int_{\Omega}D\eta(U)\cdot(\pa_t U^\e+\nabla_x\cdot \bm{A}(U^\e))\,\d x\\
		&\quad +\frac{\e^2}{2}\frac{\d}{\d t}\int_{\Omega}|\nabla\sqrt{\rho^\e}|^2\,\d x\\
		&=\frac{\d}{\d t}\int_{\Omega}\mathcal{E}^\e(t,x)\,\d x-\int_{\Omega}\nabla_x(D\eta(U)):\bm{A}(U^\e|U)\,\d x-\int_{\Omega}D\eta(U)\cdot(\pa_t U^\e+\nabla_x\cdot \bm{A}(U^\e))\,\d x\\
		&=-\int_{\Omega}\nabla_x(D\eta(U)):\bm{A}(U^\e|U)\,\d x-\int_{\Omega}D\eta(U)\cdot(\pa_t U^\e+\nabla_x\cdot \bm{A}(U^\e))\,\d x\\
		&=:\mathcal{I}_{11}+\mathcal{I}_{12},
		\end{align*}
		where we used the conservation of the total energy. In the following, we estimate each term $\mathcal{I}_{1i}$ separately.\\
		
		\noindent $\bullet$ (Estimate of $\mathcal{I}_{11}$): We first recall that the derivative of the entropy $D\eta$ and the relative flux $A(U^\e|U)$ are given as
		\[\nabla_x(D\eta(U)) = \nabla_x\begin{pmatrix}
			-\frac{|u|^2}{2}+\rho^{\gamma-1}\\u
		\end{pmatrix},\quad \bm{A}(U^\e|U) = \begin{pmatrix}
		0\\ \rho^\e(u^\e-u)\otimes (u^\e-u)+p(\rho^\e|\rho)I_2
		\end{pmatrix}.\]
		Since $u$ is smooth, $\|\nabla_x u\|_{L^\infty(0,T_*;L^\infty(\bbr^2))}$ is bounded. Therefore, we estimate $\mathcal{I}_{11}$ as
		\begin{align*}
		\mathcal{I}_{11}&=-\int_{\Omega}\nabla_x(D\eta(U)):\bm{A}(U^\e|U)\,\d x \le C\left(\int_{\Omega}\rho^\e|u^\e-u|^2\,\d x +\int_{\Omega}p(\rho^\e|\rho)\,\d x\right)\\
		&\le C\int_{\Omega}\eta(U^\e|U)\,\d x.
		\end{align*}

		\noindent $\bullet$ (Estimate of $\mathcal{I}_{12}$): To estimate $\mathcal{I}_{12}$, we first note that $U^\e$ satisfies \eqref{CSS-hydro-e}$_{1,2}$, and therefore, 
		\[\pa_t U^\e +\nabla_x\cdot \bm{A}(U^\e) = \begin{pmatrix}0\\ \rho^\e\frac{\e^2}{2}\nabla\left(\frac{\Delta \sqrt{\rho^\e}}{\sqrt{\rho^\e}}\right)\end{pmatrix}.\]
		Hence, the integrand of $\mathcal{I}_{12}$ can be computed as 
		\[D\eta(U)\cdot (\pa_t U^\e+\nabla_x\cdot \bm{A}(U^\e)) = \frac{\e^2}{2}\rho^\e u \cdot \nabla \left(\frac{\Delta \sqrt{\rho^\e}}{\sqrt{\rho^\e}}\right).\] 
		Therefore, we use the integration-by-part to estimate $\mathcal{I}_{12}$ as 
		\begin{align*}
		\mathcal{I}_{12} &= -\int_{\Omega}D\eta(U)\cdot(\pa_tU^\e+\nabla_x\cdot \bm{A}(U^\e))\,\d x = -\frac{\e^2}{2}\int_{\Omega}\rho^\e u\cdot\nabla\left(\frac{\Delta \sqrt{\rho^\e}}{\sqrt{\rho^\e}}\right)\,\d x\\
		&=\frac{\e^2}{2}\int_{\Omega}\nabla\cdot(\rho^\e u)\frac{\Delta \sqrt{\rho^\e}}{\sqrt{\rho^\e}}\,\d x = \frac{\e^2}{2}\int_{\Omega}(\nabla\cdot u)\sqrt{\rho^\e}\Delta \sqrt{\rho^\e}\,\d x +\frac{\e^2}{2}\int_{\Omega}u\cdot\nabla\rho^\e \frac{\Delta \sqrt{\rho^\e}}{\sqrt{\rho^\e}}\,\d x\\
		&= \frac{\e^2}{2}\int_{\Omega}(\nabla\cdot u)\sqrt{\rho^\e}\Delta \sqrt{\rho^\e}\,\d x +\e^2\int_{\Omega}u\cdot\nabla\sqrt{\rho^\e} \Delta \sqrt{\rho^\e}\,\d x=:\mathcal{I}_{121}+\mathcal{I}_{122}.
		\end{align*}
		
		\noindent $\diamond$ (Estimate of $\mathcal{I}_{121}$): We use integrate-by-parts once more and use the boundedness of $\nabla\cdot u$ and $\nabla\nabla\cdot u$ to estimate $\mathcal{I}_{121}$ as
		\begin{align*}
		\mathcal{I}_{121}&= \frac{\e^2}{2} \left(-\int_{\Omega}(\nabla\nabla\cdot u)\cdot \nabla\sqrt{\rho^\e} \sqrt{\rho^\e}\,\d x-\int_{\Omega}(\nabla\cdot u)|\nabla\sqrt{\rho^\e}|^2\,\d x\right)\\
		&\le C\e^2\left(\left(\int_{\Omega}|\nabla\sqrt{\rho^\e}|^2\right)^{\frac{1}{2}}\left(\int_{\Omega}\rho^\e\,\d x\right)^{\frac{1}{2}}+\int_{\Omega}|\nabla\sqrt{\rho^\e}|^2\,\d x\right)\\
		&\le C\e^2\left(\int_{\Omega}|\nabla\sqrt{\rho^\e}|^2\,\d x+1\right),
		\end{align*}
		where we use the Cauchy-Schwartz inequality and the conservation of the mass $\int_{\Omega}\rho^\e\,\d x = 1$.\\
		
		\noindent $\diamond$ (Estimate of $\mathcal{I}_{122}$): For $\mathcal{I}_{122}$, we first note that
		\begin{align*}
		\int_{\Omega}u\cdot \nabla\sqrt{\rho^\e}\Delta\sqrt{\rho^\e}\,\d x &= \sum_{i,j=1}^2 \int_{\Omega}u_i \pa_i\sqrt{\rho^\e}\pa_j^2\sqrt{\rho^\e}\,\d x \\
		&= -\sum_{i,j=1}^2\int_{\Omega}\pa_ju_i\pa_i \sqrt{\rho^\e}\pa_j\sqrt{\rho^\e}\,\d x-\sum_{i,j=1}^2 \int_{\Omega} u_i \pa_i\pa_j\sqrt{\rho^\e}\pa_j\sqrt{\rho^\e}\,\d x\\
		&=-\int_{\Omega}\nabla\sqrt{\rho^\e}\cdot \nabla u \cdot \nabla\sqrt{\rho^\e}\,\d x -\frac{1}{2}\int_{\Omega} u\cdot \nabla(|\nabla \sqrt{\rho^\e}|^2)\,\d x\\
		&=-\int_{\Omega}\nabla\sqrt{\rho^\e}\cdot \nabla u\cdot \nabla \sqrt{\rho^\e}\,\d x +\frac{1}{2}\int_{\Omega}(\nabla\cdot u)|\nabla \sqrt{\rho^\e}|^2\,\d x\\
		&\le C\int_{\Omega}|\nabla\sqrt{\rho^\e}|^2\,\d x.
		\end{align*}
		Therefore, we conclude that $\mathcal{I}_{122}\le C\e^2\int_{\Omega}|\nabla\sqrt{\rho^\e}|^2\,\d x$.\\
		
		\noindent We now combine the estimates of $\mathcal{I}_{121}$ and $\mathcal{I}_{122}$ to estimate $\mathcal{I}_{12}$ as
		\begin{align*}
		\mathcal{I}_{12}\le C\e^2\int_{\Omega}|\nabla\sqrt{\rho^\e}|^2\,\d x+C\e^2.
		\end{align*}
		Therefore, combining the estimates of $\mathcal{I}_{11}$ and $\mathcal{I}_{12}$, we conclude the estimate of the modulated energy as
		\begin{align*}
		\frac{\d \mathcal{H}^\e}{\d t}\le C\left(\int_{\Omega}\eta(U^\e|U)+\frac{\e^2}{2}\int_{\Omega}|\nabla\sqrt{\rho^\e}|^2\,\d x\right)+C\e^2\le C\mathcal{H}^\e+C\e^2.
		\end{align*}
		Then, Gr\"onwall inequality implies
		\[\mathcal{H}^\e(t) \le C(\mathcal{H}^\e(0)+\e^2)\le C\e^{\min\{\lambda,2\}},\]
		where the last inequality comes from the well-prepared initial data condition \eqref{well-prepared}.
	\end{proof}
	
	\subsection{Proof of Theorem \ref{thm:main-1}}\label{sec:4.2}
	We now present the asymptotic convergence of the SCS equations using the modulated energy estimates \eqref{est-modulated} in Proposition \ref{P4.1}. Although the convergence of $(\rho^\e,u^\e)$ can be obtained in the same way as in \cite{LW12}, we provide the full proof for the completeness of the paper. First, it is well-known (e.g. \cite{LM98}) that the $L^\gamma$-norm of $\rho^\e-\rho$ for can be controlled by the relative pressure $p(\rho^\e|\rho)$ as follows.
	\begin{lemma}\label{L4.1}
	Let $\gamma>1$ be a constant.
	\begin{enumerate}
		\item If $\gamma\ge 2$,
	\[|\rho^\e-\rho|^\gamma \le(\rho^\e)^\gamma-\rho^\gamma-\gamma\rho^{\gamma-1}(\rho^\e-\rho)=\frac{\gamma}{\gamma-1}p(\rho^\e|\rho).\]
		\item If $1<\gamma<2$,  
		\[\begin{cases}
		|\rho^\e-\rho|^2\le C\rho^{2-\gamma} p(\rho^\e|\rho),\quad &\mbox{if}\quad \rho^\e\le 2\rho,\\
		|\rho^\e-\rho|^\gamma\le Cp(\rho^\e|\rho)\quad &\mbox{if}\quad \rho^\e\ge 2\rho.
		\end{cases}\]
	\end{enumerate}
	\end{lemma}
	\begin{proof} Although the estimates are well-known, we present the detailed proof for the completeness of the paper.\\
		
	\noindent(1) It suffice to show that for $x>0$,
		\[x^\gamma-1-\gamma(x-1)\ge |x-1|^\gamma.\]
		Once we obtain the above estimate, the desired estimate directly obtained by taking $x=\frac{\rho^\e}{\rho}$. We take $f(x):=x^\gamma-1-\gamma(x-1)-|x-1|^\gamma$. Then, $f(1) = 0$ and
		\[f'(x) = \begin{cases}
			\gamma (x^{\gamma-1}-1 - (x-1)^{\gamma-1})>0,\quad&\mbox{if}\quad  x>1,\\
			\gamma (x^{\gamma-1}-1 + (1-x)^{\gamma-1})<0,\quad&\mbox{if}\quad 0<x<1.
		\end{cases}\]
	Therefore, $f(x)$ has the minimum $f(1)=0$, which implies the desired estimate.\\
	
	\noindent (2) We first show that
	\[|\rho^\e-\rho|^2\le C\rho^{2-\gamma}p(\rho^\e|\rho),\quad\mbox{if}\quad \rho^\e\le 2\rho.\]
	By dividing $\rho^2$ on both sides, it is equivalent to show that
	\[(x-1)^2\le C \left(x^\gamma -1-\gamma \left(x-1\right)\right),\quad \mbox{if} \quad x\le 2.\]
	We define $g(x):=x^\gamma-1-\gamma(x-1)$. Then, it follows from the Taylor's theorem that
	\[g(x) = g(1) +g'(1)(x-1)+\frac{g''(\xi)}{2}(x-1)^2=\frac{g''(\xi)}{2}(x-1)^2,\quad \mbox{for some $\xi$ between 1 and $x$}.\]
	Since, $x\le 2$, we also have $\xi\le 2$, and therefore,
	\[\frac{g(x)}{(x-1)^2} = \frac{g''(\xi)}{2} = \gamma(\gamma-1) \xi^{\gamma-2} \ge \gamma(\gamma-1) 2^{\gamma-2}.\]
	Thus, the desired estimate holds with $C = \frac{2^{2-\gamma}}{\gamma(\gamma-1)}$. The second assertion is equivalent to
	\[(x-1)^\gamma \le C(x^\gamma-1-\gamma (x-1)),\quad x\ge 2.\]
	Similarly, we have
	\[\frac{g(x)}{(x-1)^\gamma} = \frac{g''(\xi)}{2}(x-1)^{2-\gamma}=\frac{\gamma(\gamma-1)}{2}\xi^{\gamma-2}(x-1)^{2-\gamma}\ge \frac{\gamma(\gamma-1)}{2}\left(\frac{x-1}{x}\right)^{2-\gamma}\ge \frac{\gamma(\gamma-1)}{2}\left(\frac{1}{2}\right)^{2-\gamma}.\]
	Therefore, we still have $(x-1)^\gamma \le C(x^\gamma-1-\gamma(x-1))$ with the same $C$.
	\end{proof}

	We now show the desired convergence, according to the spatial domain.\\
	
	\noindent $\bullet$ (Case 1: $\Omega =\bbr^2$) In this case, we assume that $\gamma\ge 2$. It follows from Lemma \ref{L4.1} (1) that $\|\rho^\e-\rho\|_{L^\gamma}$ is can be bounded by $\int_{\bbr^2}p(\rho^\e|\rho)\,\d x\le \mathcal{H}^\e$, which implies
	\[\rho^\e\to\rho\quad\mbox{\textup{in}}\quad L^\gamma(\bbr^2)\quad\mbox{as}\quad\e\to0.\] 
	For the convergence regarding $u^\e$, we use the H\"older inequality and 
	\[\|u\|_{L^\frac{2\gamma}{\gamma-1}}\le \|u\|_{H^1}\le C,\] 
	to obtain
	\begin{align*}
	\|\rho^\e u^\e-\rho u\|_{L^{\frac{2\gamma}{\gamma+1}}}&\le \|\rho^\e (u^\e-u)\|_{L^\frac{2\gamma}{\gamma+1}}+\|(\rho^\e-\rho)u\|_{L^\frac{2\gamma}{\gamma+1}}\\
	&\le \|\sqrt{\rho^\e}\|_{L^{2\gamma}}\|\sqrt{\rho^\e}|u^\e-u|\|_{L^2}+\|\rho^\e-\rho\|_{L^\gamma}\|u\|_{L^\frac{2\gamma}{\gamma-1}}\\
	&\le C\|\sqrt{\rho^\e}|u^\e-u|\|_{L^2}+C\|\rho^\e-\rho\|_{L^\gamma}\le C\mathcal{H}^\e\to0.
	\end{align*}
	Furthermore, since
	\[|\sqrt{\rho^\e}-\sqrt{\rho}|^2\le |\rho^\e-\rho|,\]
	we also observe
	\begin{align*}
	\|\sqrt{\rho^\e}u^\e-\sqrt{\rho}u\|_{L^2} &\le \|\sqrt{\rho^\e}|u^\e-u|\|_{L^2} +\|(\sqrt{\rho^\e}-\sqrt{\rho})|u|\|_{L^2}\\
	&\le \|\sqrt{\rho^\e}|u^\e-u|\|_{L^2}+\|u\|_{L^{\frac{2\gamma}{\gamma-1}}}\|\sqrt{\rho^{\e}}-\sqrt{\rho}\|_{L^{2\gamma}}\\
	&\le \mathcal{H}^\e +C\|\rho^\e-\rho\|_{L^\gamma}^\frac{1}{2}\rightarrow 0.
	\end{align*}
For the convergence of $A^{\e}_0$, we recall that the difference $A^{\e}_0-A_0$ satisfies
 \begin{align*}
 	\begin{aligned}
 		\Delta (A^\e_0-A_0)= \partial_1 (\rho^\e u^{\e}_2-\rho u_2)-\partial_2(\rho^{\e}u_1^{\e}-\rho u_1),
 	\end{aligned}
 \end{align*}
which can be represented by
\begin{align*}
	A_0^{\e}-A_0=\frac{1}{2\pi}\frac{x_1}{|x|^2}*(\rho^{\e}u_2^\e-\rho u_2)-\frac{1}{2\pi}\frac{x_2}{|x|^2}*(\rho^{\e}u_1^\e-\rho u_1).
\end{align*}
By the Hardy-Littlewood-Sobolev inequality and the Calder\'{o}n-Zygmund inequality, we have
	\[
	\|A_0^{\e}-A_0\|_{L^{2\gamma}}\leq \|\rho^\e u^{\e}-\rho u\|_{L^{\frac{2\gamma}{\gamma+1}}},\quad\mbox{and}\quad  \|\nabla(A_0^{\e}-A_0)\|_{L^{\frac{2\gamma}{\gamma+1}}}\leq \|\rho^{\e}u^{\e}-\rho u\|_{L^{\frac{2\gamma}{\gamma+1}}}.\]
These, together with the above estimate of $\rho^\e u^\e$, imply
\begin{align*}
	\begin{aligned}
	\|A_0^{\e}-A_0\|_{L^{2\gamma}}\rightarrow 0,\quad \|\nabla(A_0^{\e}-A_0)\|_{L^{\frac{2\gamma}{\gamma+1}}}\rightarrow 0,\quad \mbox{as} \quad \e\rightarrow 0.	
	\end{aligned}		
\end{align*}
Similarly, to obtain the convergence of $A^{\e}$, we recall that the difference $A^{\e}-A$ satisfies
\begin{align*}
	\begin{aligned}
		\Delta(A^\e-A) = (\nabla(\rho-\rho^\e))^\perp,
	\end{aligned}
\end{align*}
and apply the Hardy-Littlewood-Sobolev inequality, the H\"older inequality and the Calder\'on-Zygmund inequality as follows
\begin{align*}
	&\|A^{\e}-A\|_{L^{2\gamma}}\leq \|\rho^\e-\rho\|_{L^{\frac{2\gamma}{\gamma+1}}}\leq (\|\sqrt{\rho^{\e}}\|_{L^2}+\|\sqrt{\rho}\|_{L^2})\|\sqrt{\rho^{\e}}-\sqrt{\rho}\|_{L^{2\gamma}}\leq C\|\rho^{\e}-\rho\|_{L^{\gamma}},
	\\
	&\|\nabla(A^{\e}-A)\|_{L^{\gamma}}\leq \|\rho^{\e}-\rho\|_{L^{\gamma}}.
\end{align*}
Finally, the above inequalities yield
\begin{align*}
	\begin{aligned}
		\|A^{\e}-A\|_{L^{2\gamma}}\rightarrow 0,\quad \|\nabla(A^{\e}-A)\|_{L^{\gamma}}\rightarrow 0,\quad \mbox{as} \quad \e\rightarrow 0.	
	\end{aligned}		
\end{align*}

\noindent $\bullet$ (Case 2: $\Omega =\bbt^2$) In this case, we assume $\gamma>1$. Since the case when $\gamma\ge 2$ is exactly the same as before, we only focus on the case when $1<\gamma<2$. In this case, we use the boundedness $\|\rho\|_{L^\infty}$ to obtain
\begin{align*}
\int_{\bbt^2}|\rho^\e-\rho|^\gamma\,\d x&=\int_{\{\rho^\e\ge 2\rho\}}|\rho^\e-\rho|^\gamma\,\d x+\int_{\{\rho^\e\le 2\rho\}}|\rho^\e-\rho|^\gamma\,\d x\\
&\le C\int_{\bbt^2}p(\rho^\e|\rho)\,\d x+ \left(\int_{\{\rho^\e \le 2\rho\}}|\rho^\e-\rho|^2\,\d x\right)^{\frac{\gamma}{2}}\left(\int_{\{\rho^\e\le 2\rho\}}1\,\d x\right)^{\frac{2-\gamma}{2}}\\
&\le C\mathcal{H}^\e(t) +C\left(\mathcal{H}^\e(t)\right)^{\frac{\gamma}{2}}\to0.
\end{align*}
After obtaining the convergence of $\|\rho^\e-\rho\|_{L^\gamma}$ for all $\gamma>1$, the remaining procedure is the same as (Case 1). In particular, the representation of the solution to the Poisson equation in the periodic domain is also available with the small modification of the Newtonian potential. This completes the proof of Theorem \ref{thm:main-1}. 

\qed

	\section{Hydrodynamic limit estimate to incompressible Euler-Poisson system}\label{sec:5}
	\setcounter{equation}{0}
	In this section, we present the estimate for the modulated energy of \eqref{CSS-e-incomp} and complete the detailed proof of Theorem \ref{thm:main-2}. We recall that the spatial domain and $\gamma$ are set to be $\Omega = \bbt^2$ and $\gamma \ge 2$.
	
	\subsection{Modulated energy estimate}
	Different from the compressible ECS equations \eqref{CSE}, the limit system \eqref{CSE-incomp} cannot be written in terms of conservation laws \eqref{conservation}. Therefore, instead of using the relative entropy estimate in Proposition \ref{prop:rel-ent}, we directly estimate the time derivative of the modulated energy. We consider the following modulated energy for the system \eqref{CSS-e-incomp}: 
	\[\widetilde{\mathcal{H}}^\e(t):=\frac{1}{2}\int_{\bbt^2}|(\e^{1-\alpha}\widetilde{D}^\e-\i u)\psi^\e|^2\,\d x+\frac{\e^{-2\alpha}}{\gamma}\int_{\bbt^2}((\rho^\e)^{\gamma/2}-1)^2\,\d x,\quad \widetilde{D}^\e = \nabla +\frac{\i}{\e^{1+\alpha}}A^\e\]
	which can be also written as the hydrodynamic variable:
	\begin{align*}
		\widetilde{\mathcal{H}}^\e(t)&:=\widetilde{\mathcal{E}}^\e(t)-\int_{\bbt^2}\rho^\e u^\e\cdot u\,\d x+\frac{1}{2}\rho^\e|u|^2\,\d x\\
		&=\frac{1}{2}\int_{\bbt^2}\rho^\e|u^\e-u|^2\,\d x+\frac{\e^{-2\alpha}}{\gamma}\int_{\bbt^2}((\rho^\e)^{\gamma/2}-1)^2\,\d x+\frac{\e^{2-2\alpha}}{2}\int_{\bbt^2}|\nabla\sqrt{\rho^\e}|^2\,\d x.
	\end{align*}
	Here, $u$ is the solution to the incompressible Euler equation:
	\[\pa_t u+ (u\cdot \nabla)u +\nabla\pi =0,\quad \nabla\cdot u=0,\quad u|_{t=0} = u_{\textup{in}}.\]
	
	\begin{lemma}\label{L5.1}
		Let $(\psi^\e,A^\e_0,A^\e)$ be the unique global solution to the SCS equations \eqref{CSS-e} subject to the initial data $\psi^\e_{\textup{in}}\in H^2(\Omega)$ and  $A^\e_{\textup{in}}$ satisfying \eqref{compatibility}  and let $u$ be the unique local-in-time smooth solution to the incompressible Euler equations \eqref{CSE-incomp}$_{1}$ subject to the initial data $u_{\textup{in}}$. Suppose that the initial data satisfy the well-prepared condition \eqref{well-prepared-2}. Then,
		\[\widetilde{\mathcal{H}}^\e(t)\le C\e^{\min\left\{\frac{2\alpha}{\gamma},\lambda\right\}},\quad 0\le t< T_*. \]
	\end{lemma}
	\begin{proof}
		First of all, it follows from the conservation of the energy (Proposition \ref{P2.2}) that
		\[\frac{\e^{-2\alpha}}{\gamma}\int_{\bbt^2}((\rho^\e)^{\gamma/2}-1)^2\,\d x\le \widetilde{\mathcal{E}}(t)=\widetilde{\mathcal{E}}(0).\]
		Since $|\rho^\e-1|^{\gamma}\le |(\rho^\e)^{\frac{\gamma}{2}}-1|^2$, we obtain $\|\rho^\e-1\|_{L^\gamma}\le C\e^{\frac{2\alpha}{\gamma}}$. We now use \eqref{CSS-hydro-e-2} to estimate $\widetilde{\mathcal{H}}^\e$ as
	\begin{align*}
		\frac{\d}{\d t} \widetilde{\mathcal{H}}^\e(t)&=-\frac{\d}{\d t}\int_{\bbt^2}\rho^\e u^\e\cdot u \,\d x+\frac{1}{2}\frac{\d}{\d t}\int_{\bbt^2}\rho^\e|u|^2\,\d x\\
		&=\int_{\bbt^2}	\left(\nabla\cdot(\rho^\e u^\e\otimes u^\e)+\e^{-2\alpha}\nabla p(\rho^\e)-\frac{\e^{2-2\alpha}\rho^\e}{2}\nabla\left(\frac{\Delta\sqrt{\rho^\e}}{\sqrt{\rho^\e}}\right)\right)\cdot u\,\d x\\
		&\quad+\int_{\bbt^2}\rho^\e u^\e\cdot ((u\cdot \nabla)u+\nabla \pi)\,\d x\\
		&\quad +\frac{1}{2}\int_{\bbt^2}\rho^\e u^\e\cdot \nabla(|u|^2)\,\d x+\int_{\bbt^2}\rho^\e u\cdot(-(u\cdot \nabla) u-\nabla\pi)\,\d x\\
		&=:\mathcal{I}_{21}+\mathcal{I}_{22}+\mathcal{I}_{23}.
	\end{align*}
	
	\noindent $\bullet$ (Estimate of $\mathcal{I}_{21}$): We use the integration by part to obtain 
	\begin{align*}
		\mathcal{I}_{21} &= -\int_{\bbt^2}\rho^\e (u^\e\otimes u^\e):\nabla u\,\d x +\frac{\e^{-2\alpha}}{2}\int_{\bbt^2}\nabla p(\rho^\e)\cdot u\,\d x +\frac{\e^{2-2\alpha}}{2}\int_{\bbt^2}\frac{\Delta \sqrt{\rho^\e}}{\sqrt{\rho^\e}}\nabla\rho^\e\cdot u\,\d x\\
		&=-\int_{\bbt^2}\rho^\e (u^\e\otimes u^\e):\nabla u\,\d x+\e^{2-2\alpha}\int_{\bbt^2} u\cdot \nabla\sqrt{\rho^\e}\Delta\sqrt{\rho^\e}\,\d x,
	\end{align*}
	where we used the incompressibility condition $\nabla\cdot u=0$.\\
	
	\noindent $\bullet$ (Estimate of $\mathcal{I}_{22}$): It is straightforward that
	\[\mathcal{I}_{22}= \int_{\bbt^2}\rho^\e (u^\e\otimes u):\nabla u\,\d x+\int_{\bbt^2}\rho^\e u^\e\cdot \nabla \pi\,\d x.\]
	
	\noindent $\bullet$ (Estimate of $\mathcal{I}_{23}$): Similarly, $\mathcal{I}_{23}$ can also be written as
	\[\mathcal{I}_{23}= \int_{\bbt^2}\rho^\e (u\otimes u^\e):\nabla u\,\d x-\int_{\bbt^2}\rho^\e (u\otimes u):\nabla u\,\d x-\int_{\bbt^2}\rho^\e u\cdot \nabla \pi\,\d x.\]
	
	Therefore, we combine the estimates for $\mathcal{I}_{21}$, $\mathcal{I}_{22}$ and $\mathcal{I}_{23}$ to obtain
	
	\begin{align*}
		\frac{\d}{\d t}\widetilde{\mathcal{H}}^\e(t)&=-\int_{\bbt^2}\rho^\e((u^\e-u)\otimes (u^\e-u)):\nabla u\,\d x+\int_{\bbt^2}\rho^\e(u^\e-u)\cdot\nabla\pi\,\d x\\
		&\quad+\e^{2-2\alpha}\int_{\bbt^2}u\cdot\nabla \sqrt{\rho^\e}\Delta \sqrt{\rho^\e}\,\d x\\
		&\le C\int_{\bbt^2}\rho^\e|u^\e-u|^2\,\d x +C\e^{2-2\alpha}\int_{\bbt^2}|\nabla\sqrt{\rho^\e}|^2\,\d x+\int_{\bbt^2}\rho^\e(u^\e-u)\cdot\nabla\pi\,\d x.
	\end{align*}
	We now estimate the last term as
	\begin{align*}
		\int_{\bbt^2}\rho^\e(u^\e-u)\cdot\nabla\pi\,\d x&=\int_{\bbt^2}\rho^\e u^\e \cdot\nabla \pi \,\d x-\int_{\bbt^2}\rho^\e u\cdot\nabla\pi\,\d x\\
		&=\int_{\bbt^2}\pa_t(\rho^\e-1)\pi\,\d x-\int_{\bbt^2}(\rho^\e-1)u\cdot\nabla\pi\,\d x\\
		&=\frac{\d}{\d t}\int_{\bbt^2}(\rho^\e-1)\pi\,\d x-\int_{\bbt^2}(\rho^\e-1)\pa_t \pi\,\d x-\int_{\bbt^2}(\rho^\e-1)u\cdot\nabla\pi\,\d x\\
		&=\frac{\d}{\d t}\int_{\bbt^2}(\rho^\e-1)\pi\,\d x-\int_{\bbt^2}(\rho^\e-1)(\pa_t\pi + u\cdot\nabla\pi)\,\d x\\
		&\le \frac{\d}{\d t}\int_{\bbt^2}(\rho^\e-1)\pi\,\d x+ \|\rho^\e-1\|_{L^\gamma}\|\pa_t\pi+u\cdot\nabla\pi\|_{L^{\frac{\gamma}{\gamma-1}}}\\
		&\le \frac{\d}{\d t}\int_{\bbt^2}(\rho^\e-1)\pi\,\d x+ C\e^{\frac{2\alpha}{\gamma}}.
	\end{align*}
	Therefore, we conclude that
	\[\frac{\d\widetilde{\mathcal{H}}^\e}{\d t}\le \frac{\d}{\d t}\int_{\bbt^2}(\rho^\e-1)\pi\,\d x+ C\widetilde{\mathcal{H}}^\e(t) +C\e^{\frac{2\alpha}{\gamma}},\quad 0\le t< T_*\]
	and by integrating over $[0,t]$, we deduce
	\begin{align*}
	\widetilde{\mathcal{H}}^\e(t) &\le \widetilde{\mathcal{H}}^\e(0)+\int_{\bbt^2}(\rho^\e(t,x)-1)\pi(t,x)\,\d x-\int_{\bbt^2}(\rho^\e(0,x)-1)\pi(0,x)\,\d x+C\int_0^t \widetilde{\mathcal{H}}^\e(s)\,\d s +C\e^{\frac{2\alpha}{\gamma}}\\
	&\le \widetilde{\mathcal{H}}^\e(0)+\|\rho^\e(t)-1\|_{L^\gamma}\|\pi(t)\|_{L^\frac{\gamma}{\gamma-1}} +\|\rho^\e(0)-1\|_{L^\gamma}\|\pi(0)\|_{L^\frac{\gamma}{\gamma-1}}+C\int_0^t \widetilde{\mathcal{H}}^\e(s)\,\d s +C\e^{\frac{2\alpha}{\gamma}}\\
	&\le \widetilde{\mathcal{H}}^\e(0)+C\e^{\frac{2\alpha}{\gamma}}+C\int_0^t\widetilde{\mathcal{H}}^\e(s)\,\d s,
	\end{align*}
	where we used $\|\rho^\e-1\|_{L^\gamma} \le C\e^{\frac{2\alpha}{\gamma}}$. Then, Gr\"ownall inequality implies the desired estimate:
	\[\widetilde{\mathcal{H}}^\e(t)\le C\left(\widetilde{\mathcal{H}}^\e(0)+\e^{\frac{2\alpha}{\gamma}}\right)\le C\e^{\min\left\{\frac{2\alpha}{\gamma},\lambda\right\}},\quad 0\le t< T_*.\]
	\end{proof}
	\subsection{Proof of Theorem \ref{thm:main-2}.}
	We now provide the desired convergence estimate and complete the proof of Theorem \ref{thm:main-2}. We already observe in Lemma \ref{L5.1} that
	\[\|\rho^\e-1\|_{L^\gamma(\bbt^2)}\le C\e^{\frac{2\alpha}{\gamma}}\to0,\quad \mbox{as}\quad \e\to0.\]
	Moreover, since both $\rho^\e-1$ and $1$ are in $L^\gamma(\bbt^2)$, $\rho^\e\in L^\gamma(\bbt^2)$ and $\|\rho^\e\|_{L^\gamma}$ is uniformly bounded. Therefore, we estimate the difference $\|\rho^\e u^\e-u\|_{L^{\frac{2\gamma}{\gamma+1}}}$ as 
	\begin{align}
		\begin{aligned}\label{E-1}
		\|\rho^\e u^\e-u\|_{L^{\frac{2\gamma}{\gamma+1}}}&\le \|\rho^\e(u^\e-u)\|_{L^{\frac{2\gamma}{\gamma+1}}}+\|(\rho^\e-1)u\|_{L^\frac{2\gamma}{\gamma+1}}\\
		&\le \|\sqrt{\rho^\e}\|_{L^{2\gamma}}\|\sqrt{\rho^\e}|u^\e-u|\|_{L^2}+\|\rho^\e-1\|_{L^\gamma}\|u\|_{L^{\frac{2\gamma}{\gamma-1}}}\\
		&\le C\widetilde{\mathcal{H}}^\e(t) +C\e^{\frac{2\alpha}{\gamma}}\to 0,
		\end{aligned}
	\end{align}
	as well as the difference $\|\sqrt{\rho^\e}u^\e-u\|_{L^2}$ as
	\begin{align*}
		\|\sqrt{\rho^\e}u^\e-u\|_{L^2} &\le \|\sqrt{\rho^\e}|u^\e-u|\|_{L^2} +\|(\sqrt{\rho^\e}-1)|u|\|_{L^2}\\
		&\le \|\sqrt{\rho^\e}|u^\e-u|\|_{L^2}+\|u\|_{L^{\frac{2\gamma}{\gamma-1}}}\|\sqrt{\rho^{\e}}-1\|_{L^{2\gamma}}\\
		&\le C\widetilde{\mathcal{H}}^\e +C\|\rho^\e-1\|_{L^\gamma}^\frac{1}{2}\rightarrow 0.
	\end{align*}
	For the convergence of the Chern-Simons part, we first note that $(A_0^\e,A^\e)$ satisfies
	\[ \Delta A_0^\e = \nabla\times (\rho^\e u^\e),\quad \Delta A^{\varepsilon}=-\e^{\alpha}(\nabla \rho^\e)^{\perp}\]
	while $A_0$ satisfies
	\[\Delta A_0 =\nabla\times u.\]
	Then, the desired convergences from $(A^\e_0,A^\e)$ to $(A_0,0)$ can be obtained by using exactly the same arguments with $\rho\equiv1$ in Section \ref{sec:4.2} and the convergence of $\rho^\e u^\e \to u$ in \eqref{E-1}, and this complete the proof of Theorem \ref{thm:main-2}. 
	
	\qed
	
	\section{Conclusion}\label{sec:6}
	\setcounter{equation}{0}
	In this paper, we studied the hydrodynamic limit problem for the Schr\"odinger-Chern-Simons equations. We adopt the Madelung transformation and different scalings to derive the compressible and incompressible Euler equations coupled with the Chern-Simons equations and Poisson equation respectively. To obtain the rigorous convergence, we focus on the modulated energy functional. For the case of compressible Euler limit, we relate the modulated energy with the standard relative entropy in the kinetic theory and use the well-known estimates on the relative entropy functional to provide the simple estimate of the modulated energy functional. On the other hand, for the case of the incompressible Euler limit, we directly estimate the modified energy functional to obtain the desired convergence. Beyond the Schr\"odinger-Chern-Simons equations, other interesting systems, such as Schr\"odinger-Maxwell, Chern-Simons-Higgs, or Maxwell-Klein-Gordon equations, are also actively studied from the various points of view. Therefore, investigating the hydrodynamic limits of those systems would be left as interesting future works.

	\bibliographystyle{amsplain}
	
\end{document}